\documentclass[10pt]{amsart}
\usepackage{graphicx, color}
\usepackage{amssymb}
\usepackage{amsmath}
\usepackage[T1]{fontenc}
\usepackage[displaymath, mathlines, running]{lineno}
\newtheorem{theorem}{Theorem}[section]
\newtheorem{lemma}[theorem]{Lemma}

\newtheorem{proposition}[theorem]{Proposition}
\newtheorem{remark}[theorem]{Remark}
\newtheorem{definition}[theorem]{Definition}
\newenvironment{proof of main prop}{{\bf Proof of Proposition \ref{main prop}.}}{\hfill\fbox{}\par\vspace{.2cm}}
\newenvironment{proof of main thm 1}{{\bf Proof of Theorem \ref{main thm 1}.}}{\hfill\fbox{}\par\vspace{.2cm}}
\newenvironment{proof of main thm 2}{{\bf Proof of Theorem \ref{main thm 2}.}}{\hfill\fbox{}\par\vspace{.2cm}}
\newenvironment{proof of main high 1}{{\bf Proof of Theorem \ref{main high 1}.}}{\hfill\fbox{}\par\vspace{.2cm}}
\newenvironment{proof of main high 2}{{\bf Proof of Theorem \ref{main high 2}.}}{\hfill\fbox{}\par\vspace{.2cm}}
\numberwithin{equation}{section}

\def\charf {\mbox{{\text 1}\kern-.24em {\text l}}}

\def\bea{\begin{eqnarray*}}
\def\eea{\end{eqnarray*}}
\def\be{\begin{eqnarray}}
\def\ee{\end{eqnarray}}

\begin{document}

\title[Marginally outer trap surface with capillary]{Marginally outer trap surface with capillary boundary and rigidity of initial data sets}


\author{Sanghun Lee}
\address{Department of Mathematics and Institute of Mathematical Science, Pusan National University, Busan 46241, Korea}
\email{kazauye@pusan.ac.kr}




\date{\today}

\dedicatory{}

\begin{abstract}
In this paper, we present the several rigidity results of initial data sets with boundary when a marginally outer trap surface (MOTS) with capillary boundary is embedded. First, we establish estimates for the area of a MOTS with capillary boundary. Next, we prove a rigidity for 3-dimensional initial data sets with boundary. Furthermore, we extend our results from the 3-dimensional case to the high dimensional initial data sets using the Yamabe constant.
\end{abstract}

\maketitle
\section{Introduction}

In differential geometry, understanding the relationship between curvature and topology is a fundamental question. In \cite{SY}, Schoen and Yau showed that any oriented, two-sided, stable minimal closed surface $\Sigma$ in a $3$-dimensional Riemannian manifold $(M^{3}, g)$ with positive scalar curvature must have genus zero. Subsequently, Fischer-Colbrie and Schoen \cite{FCS} proved that in the case of non-negative scalar curvature, any stable minimal closed surface in a $3$-dimensional Riemannian manifold must have genus zero or one. Motivated by these results, Cai and Galloway \cite{CG} investigated the rigidity of 3-dimensional Riemannian manifolds with non-negative scalar curvature. Their main result is as follows.

\begin{theorem}[Cai and Galloway \cite{CG}]
Let $(M,g)$ be a $3$-dimensional Riemannian manifold with non-negative scalar curvature $R^{M}$. If $\Sigma^{2}$ is a two-sided, embedded, area-minimizing $2$-torus in $(M, g)$, then $M$ is locally isometric to the product $\left((-\epsilon,\epsilon) \times \Sigma, dt^{2} + \gamma\right)$, where $\gamma = g\vert_{\Sigma}$ is the metric on $\Sigma$ induced from $M$, is flat. 
\end{theorem}

For the cases of positive and negative scalar curvature,  the results were proved by Bray, Brendle, and Neves \cite{BBN}, and by Nunes \cite{N}, respectively. On the other hand, Micallef and Moraru \cite{MM} provided a new, unified proof encompassing all cases. For high dimensional Riemannian manifolds, related results can be found in \cite{BCBS, CG2, C, DE, M3, MO, MO2}.

An interesting question in this line of research is whether the above rigidity results can be extended to manifolds with boundary. In \cite{A}, Ambrozio presented a rigidity result for $3$-dimensional Riemannian manifolds with non-empty boundary when a surface with free boundary is embedded. Here, a \textit {free boundary} means that $\Sigma$ meets $\partial M$ orthogonally along $\partial\Sigma$. Subsequently, Longa \cite{LON} generalized Ambrozio’s result to the more inclusive setting of surfaces with capillary boundary, where the surface $\Sigma$ meets $\partial M$ at a constant angle $\theta \in (0,\pi)$ along $\partial\Sigma$. The precise statement is given below.

\begin{theorem}[Longa \cite{LON}]
Let $(M, g)$ be a Riemannian $3$-manifold with non-empty boundary, and assume $R^{M}$ and $H^{\partial M}$ are bounded from below. If $\Sigma^{2}$ is a compact, two-sided, properly embedded, energy-minimizing surface with capillary boundary and contact angle $\theta \in (0,\pi)$, then
\begin{align*}
\frac{A(\Sigma)}{2}\inf_{M}\, R^{M} + \frac{A(\partial\Sigma)}{{\rm sin}\, \theta}\inf_{\partial M}\, H^{\partial M} \leq 2\pi\chi(\Sigma),
\end{align*}
where $A(\Sigma)$ is the area of $\Sigma$, $A(\partial\Sigma)$ is the length of $\partial\Sigma$, and $\chi(\Sigma)$ is the Euler characteristic of $\Sigma$, respectively. Moreover, if equality holds and one of the following hypotheses holds:
\begin{itemize}
\item[(1)] Each component of $\partial \Sigma$ is locally length-minimizing in $\partial M$; or

\item[(2)] $\inf_{\partial M}\, H^{\partial M}$ = $0$,

\end{itemize}
then either $\theta = \frac{\pi}{2}$ or $\Sigma$ is a flat and totally geodesic cylinder, $M$ is flat and $\partial M$ is totally geodesic around $\Sigma$. In the first case, there is a neighborhood of $\Sigma$ in $M$ that is isometric to $\left((-\epsilon,\epsilon) \times \Sigma, dt^{2} + \gamma\right)$, where $\gamma$ is the induced metric on $\Sigma$. Furthermore, the Gaussian curvature is constant equal to $K^{\Sigma} = \frac{1}{2}\inf_{M}R^{M}$ and the geodesic curvature is constant equal to $k^{\partial\Sigma} = \inf_{\partial M}H^{\partial M}$ in $\Sigma$.
\end{theorem}

A detailed explanation of the capillary boundary condition is provided in Section 2. 

On the other hand, in the case of high dimensional Riemannian manifolds with boundary, Ambrozio's result was extended by Barros and Cruz \cite{BC} and Longa's result was extended by Lee, Park, and Pyo \cite{LPP}.

From the perspective of general relativity, the above results can be seen as a time-symmetric case of an initial data set. An \textit{initial data set} $(M, g, h^{M})$ in general relativity consists of a Riemannian manifold $(M, g)$ and a symmetric 2-tensor $h^{M}$. These initial data sets correspond to spacelike hypersurfaces in a time-oriented Lorentzian manifold $(\mathsf{L}, \bar{g})$. In particular, a time-symmetric initial data set is one in which $h^{M} = 0$. Therefore, all of the above results can be interpreted as applying to time-symmetric initial data sets in general relativity. In the general (non-time-symmetric) setting, minimal surfaces and scalar curvature are replaced by \textit{marginally outer trapped surfaces} (MOTS) and the \textit{dominant energy condition}, respectively. A detailed description of MOTS and the dominant energy condition is provided in Section 2. 

In \cite{GM}, Galloway and Mendes extended the rigidity result for 3-dimensional Riemannian manifolds with positive scalar curvature, originally due to Bray, Brendle, and Neves \cite{BBN}, to 3-dimensional initial data sets by employing marginally outer trapped surfaces (MOTS). Their result is as follows.

\begin{theorem}[Galloway and Mendes \cite{GM}]
Let $(M, g, h^{M})$ be a $3$-dimensional initial data set. Let $\Sigma^{2}$ be a spherical MOTS in $(M, g, h^{M})$ which is weakly outermost and outer area-minimizing. Assume that $\mu - \vert J \vert \geq C_{+}$ on $M_{+}$ for $C_{+} > 0$. Then
\begin{align*}
A(\Sigma)C_{+} \leq 4\pi.
\end{align*}
If equality holds, then we have the following:
\begin{itemize}
\item[(a)] An outer neighborhood $V \approx [0,\epsilon) \times \Sigma$ of $\Sigma$ in $M$ is isometric to $\left([0,\epsilon) \times \Sigma, dt^{2} + \gamma\right)$, where $(\Sigma, \gamma)$ is the $2$-sphere of the constant Gaussian curvature $K^{\Sigma} = C_{+}$.
\item[(b)] Each slice $\Sigma_{t} \approx \{t\} \times \Sigma$ is totally geodesic as a submanifold of spacetime. Equivalently, $\chi^{+}(t) = \chi^{-}(t) = 0$, where $\chi^{\pm}(t)$ are the null second fundamental forms of $\Sigma_{t}$ in $\mathsf{L}^{4}$.
\item[(c)] $h^{M}\vert_{T\Sigma_{t}} = 0$ and $h^{M}(N_{t}, \cdot)\vert_{T\Sigma_{t}} = 0$, where $N_{t}$ is the outer unit normal of $\Sigma_{t}$, and $J = 0$ on $V$. 
\end{itemize}
\end{theorem}

Here, $\mu$ and $J$ means that $\mu = T(N_{M}, N_{M})$ and $J = T(N_{M},\cdot)\vert_{T_{p}M}$ for $p \in M$, where the Einstein tensor $G_{\bar{g}} \equiv Ric_{\bar{g}} - \frac{1}{2}R_{\bar{g}}\bar{g} = T$ in $(\mathsf{L}^{4},\bar{g})$ and $N_{M}$ is the future-directed timelike unit normal in $(M, g, h^{M})$.

Subsequently, Mendes \cite{M} generalized the rigidity results of Nunes \cite{N} and Moraru \cite{MO, MO2}, originally formulated for Riemannian manifolds with negative scalar curvature, to $(n+1)$-dimensional initial data sets.

Recently, Alaee, Lesourd, and Yau \cite{ALY} extended the concept of minimal surfaces with free boundary to marginally outer trapped surfaces (MOTS) with free boundary, and established a rigidity result for 3-dimensional initial data sets with boundary. More recently, de Almeida and Mendes \cite{AM, M2} generalized this rigidity result to $(n+1)$-dimensional initial data sets with boundary.

In this paper, we generalize the results concerning marginally outer trapped surfaces (MOTS) with free boundary in \cite{ALY, AM, M2} to the setting of MOTS with capillary boundary. More precisely, we establish a rigidity result for 3-dimensional initial data sets with boundary when a MOTS with capillary boundary is embedded. Furthermore, we extend this rigidity result from the 3-dimensional case to high dimensional initial data sets with boundary.

First, we derive an area estimate for marginally outer trapped surfaces (MOTS) with capillary boundary.

\begin{theorem} [Theorem \ref{main thm 1-1}] \label{main thm 1}
Let $(M, g, h^{M})$ be a 3-dimensional initial data set with boundary and $\Sigma^{2}$ be a compact stable MOTS with capillary boundary in $(M, g, h^{M})$. Assume that $\mu + J(N) \geq C$ for $C \in \mathbb{R}$ and $(M, g, h^{M})$ satisfies the tilted dominant boundary energy condition. Then
\begin{align} \label{area est}
A(\Sigma)C \leq 2\pi\chi(\Sigma),
\end{align}
where $\chi(\Sigma)$ is the Euler characteristic. If equality holds in (\ref{area est}), then we have
\begin{itemize}
\item[(1)] $\mu + J(N) = C$ and $\chi^{+} = 0$ on $\Sigma$,
\item[(2)]the Gaussian curvature $K^{\Sigma} = C$ on $\Sigma$ and the geodesic curvature $k^{\partial\Sigma} = 0$ along $\partial\Sigma$ in $\Sigma$,
\item[(3)] $\lambda_{1}(L_{s}) = \lambda_{1}(L) = 0$ on $\Sigma$,
\item[(4)] $H^{\partial M} + (\cos\theta) {\rm tr}_{\partial M}h^{M} = \sin\theta\vert h^{M}(\bar{N},\cdot)^{T} \vert$.
\end{itemize}
\end{theorem}

The following result is a rigidity result for a 3-dimensional initial data set with boundary. Before presenting the result, we introduce some notation and definitions necessary for its statement. First, we say that $h^{M}$ is \textit{n-convex} on $M$ if ${\rm tr}_{\pi}h^{M} \geq 0$ for all $n$-dimensional linear subspaces $\pi \subset T_{p}M$, for every $p \in M$ (see \cite{AM, M, M2}). Second, we denote by $M_{+}$ the region consisting of $\Sigma$ and the region outside $\Sigma$. Third, we say that $\Sigma$ is \textit{weakly outermost} if there exists no outer trap surfaces $(\Theta^{+} < 0)$ in $M_{+}$ that is homologous to $\Sigma$. Fourth, $\Sigma$ is said to be \textit{outer energy minimizing} if its energy is less than or equal to that of any surface in an outer neighborhood of $\Sigma$. Finally, the \textit{tilted dominant boundary energy condition} means that $H^{\partial M} + (\cos\theta){\rm tr}_{\partial M}h^{M} \geq \sin\theta\vert h^{M}(\bar{N}, \cdot)^{T} \vert$, where $h^{M}(\bar{N},\cdot)^{T}$ is the component tangential to $\partial M$ of the 1-form $h^{M}(\bar{N},\cdot)$ (see \cite{CH, CW}). A detailed description of these notions is provided in Section 2.

 We are now ready to state our main result.

\begin{theorem} [Theorem \ref{main thm 1-2}] \label{main thm 2}
Let $(M, g ,h^{M})$ be a 3-dimensional initial data sets with boundary and $\Sigma^{2}$ be a compact stable MOTS with capillary boundary in $(M, g, h^{M})$ that is weakly outermost in $M_{+}$. Assume that $\mu - \vert J \vert \geq C$ for $C \in \mathbb{R}$ and $(M, g, h^{M})$ satisfies the tilted dominant boundary energy condition, all of which hold on $M_{+}$. Suppose that the equality holds in (\ref{area est}).
\begin{itemize}
\item[(I)] If $C > 0$ and ${\rm tr}_{\pi}h^{M} \leq 0$ with $\theta \in [\frac{\pi}{2},\pi)$, then $\theta$ must be $\frac{\pi}{2}$.
\item[(II)] If $C < 0$ and 2-convex with $\theta \in (0,\frac{\pi}{2}]$, then $\theta$ must be $\frac{\pi}{2}$.
\item[(III)] If $C = 0$, 2-convex, and an energy-minimizing, then $\theta = \frac{\pi}{2}$ or not.
\end{itemize}
Moreover, the case of $\theta = \frac{\pi}{2}$, the following hold:
\begin{itemize}
\item[(i)] There exists an outer neighborhood $V$ of $\Sigma$ in $M$ such that $(V,g\vert_{V})$ is isometric to $([0,\epsilon) \times \Sigma, dt^{2} + \gamma)$, where $\gamma = g|_{\Sigma}$ and $(\Sigma,\gamma)$ have the constant Gaussian curvature $K^{\Sigma} = C$ and zero geodesic curvature $k^{\partial\Sigma} = 0$;
\item[(ii)] $h^{M} = a dt^{2}$ on $V$, where $a \in C^{\infty}$ depends only on $t \in [0,\epsilon)$;
\item[(iii)] $\mu = C$ and $J = 0$ on $V$;
\item[(iv)] The tilted dominant boundary condition is saturated along $V \cap \partial M$.
\end{itemize}
Otherwise, if without loss of generality of {\rm (III)}, $H^{\partial M} \geq 0$, and $\theta \neq \frac{\pi}{2}$, then $\Sigma$ is flat with $k^{\partial\Sigma} = 0$ and $M$ is flat with $H^{\partial M} = 0$ in an outer neighborhood of $\Sigma$.
\end{theorem}

\begin{remark}
If $\mu - \vert J \vert \geq C$, then it also follows that $\mu + J(N) \geq C$ for $C \in \mathbb{R}$, since $\mu + J(N) \geq \mu - \vert J \vert$. Therefore, the inequality (\ref{area est}) in Theorem \ref{main thm 1} remains valid under the assumption $\mu - \vert J \vert \geq C$ for $C \in \mathbb{R}$.
\end{remark}

Next, we extend Theorems \ref{main thm 1} and \ref{main thm 2} to high dimensional initial data sets. The following result provides an area estimate for high dimensional MOTS with capillary boundary.

\begin{theorem} [Theorem \ref{main high 1-1}] \label{main high 1}
Let $(M,g,h^{M})$ be an $(n+1)$-dimensional initial data set with boundary and $\Sigma^{n}$ be a compact stable MOTS with capillary boundary in $(M,g,h^{M})$. Assume that $\mu + J(N) \geq D$ for $D \in \mathbb{R}^{-} \cup \{0\}$ and $(M,g,h^{M})$ satisfies the tilted dominant boundary energy condition.
\begin{itemize}
\item[(A)] If $D < 0$ and $\sigma^{1,0}(\Sigma,\partial\Sigma) < 0$, then
\begin{align} \label{high area est}
2DA(\Sigma)^{\frac{2}{n}} \leq \sigma^{1,0}(\Sigma,\partial\Sigma).
\end{align}
\item[(B)] If $D = 0$ and $\sigma^{1,0}(\Sigma,\partial\Sigma) \leq 0$, then $\sigma^{1,0}(\Sigma,\partial\Sigma) = 0$.
\end{itemize}
Moreover, if equality holds in (\ref{high area est}) or $\sigma^{1,0}(\Sigma,\partial\Sigma) = 0$, then we have
\begin{itemize}
\item[(1)] $\mu + J(N) = D$ and $\chi^{+} = 0$ on $\Sigma$,
\item[(2)] $\Sigma$ is Einstein manifold with $R^{\Sigma} = 2D$ and totally geodesic boundary,
\item[(3)] $\lambda_{1}(L_{s}) = \lambda_{1}(L) = 0$ on $\Sigma$,
\item[(4)] $H^{\partial M} + (\cos\theta){\rm tr}_{\partial M}h^{M} = \sin\theta\vert h^{M}(\bar{N},\cdot)^{T} \vert$ along $\partial\Sigma$.
\end{itemize}
\end{theorem}

Here, $\sigma^{1,0}(\Sigma,\partial\Sigma)$ is called the \textit{Yamabe constant}, and it is a topological invariant. For closed Riemannian manifolds, the Yamabe constant was independently introduced by Kobayashi \cite{KO} and Schoen \cite{SC2, SC}. In the case of Riemannian manifolds with boundary, it was introduced by Escobar \cite{ES2}. Further details are provided in Section 5.

Based on Theorem \ref{main high 1}, we generalize the 3-dimensional rigidity result of Theorem \ref{main thm 2} to high dimensional initial data sets, and simultaneously extend the results of de Almeida and Mendes (Theorem B in \cite{AM} and Theorem 5.1 in \cite{M2}) to the capillary boundary case.

\begin{theorem} [Theorem \ref{main high 1-2}] \label{main high 2}
Let $(M,g,h^{M})$ be an $(n+1)$-dimensional initial data sets with boundary and $\Sigma^{n}$ be a compact stable MOTS with capillary boundary in $(M,g,h^{M})$ that is weakly outermost in $M_{+}$. Suppose that $\mu - \vert J \vert \geq D$ for $D \in \mathbb{R}^{-} \cup \{0\}$ and $(M,g,h^{M})$ satisfies the tilted dominant boundary energy condition, all of which hold on $M_{+}$.
\begin{itemize}
\item[(A)] If $D < 0$, $\sigma^{1,0}(\Sigma,\partial\Sigma) < 0$, then
\end{itemize}
\begin{align*}
2DA(\Sigma)^{\frac{2}{n}} \leq \sigma^{1,0}(\Sigma,\partial\Sigma).
\end{align*}
Moreover, if equality holds and $n$-convex with $\theta \in (0,\frac{\pi}{2}]$ , then $\theta$ must be $\frac{\pi}{2}$.
\begin{itemize}
\item[(B)] If $D = 0$, $\sigma^{1,0}(\Sigma,\partial\Sigma) \leq 0$, $n$-convex, and an energy-minimizing, then $\theta = \frac{\pi}{2}$ or not.
\end{itemize}
Furthermore, the case of $\theta = \frac{\pi}{2}$, the following hold:
\begin{itemize}
 \item[(i)] There exists an outer neighborhood $V$ of $\Sigma$ in $M$ such that $(V,g\vert_{V})$ is isometric to $([0,\epsilon) \times \Sigma, dt^{2} + \gamma)$, where $\gamma = g|_{\Sigma}$ and $(\Sigma,\gamma)$ is the Einstein manifold with the scalar curvature $R^{\Sigma} = 2D$ and a totally geodesic boundary;
\item[(ii)] $h^{M} = a dt^{2}$ on $V$, where $a \in C^{\infty}$ depends only on $t \in [0,\epsilon)$;
\item[(iii)] $\mu = C$ and $J = 0$ on $V$;
\item[(iv)] The tilted dominant boundary condition is saturated along $V \cap \partial M$.
\end{itemize}
Otherwise, if without loss of generality of {\rm (B)}, $Ric^{M} = \frac{\mu}{n+1}g$, $H^{\partial M} \geq 0$, and $\theta \neq \frac{\pi}{2}$, then $\Sigma$ is Ricci flat with totally geodesic boundary and $M$ is Ricci flat with $H^{\partial M} = 0$ in an outer neighborhood of $\Sigma$.
\end{theorem}

\section{Preliminary}

In this section, we introduce the notions of initial data sets, marginally outer trapped surfaces (MOTS), capillary boundary, and the Yamabe constant.

Let $(\mathsf{L},\bar{g})$ be an $(n+2)$-dimensional spacetime and time-oriented Lorentzian manifold and satisfies the Einstein field equation
\begin{align*}
G_{\bar{g}} \equiv Ric_{\bar{g}} - \frac{1}{2}R_{\bar{g}}\bar{g} = T,
\end{align*}
where $Ric_{\bar{g}}$ and $R_{\bar{g}}$ are the Ricci and scalar curvature of the metric $\bar{g}$, and $T$ is a symmetric two tensor, respectively. 

Let $M$ be an $(n+1)$-dimensional spacelike hypersurface embedded in $(\mathsf{L}, \bar{g})$, and let $N_{M}$ denote the future-directed timelike unit normal vector of $M$. Then the corresponding \textit{initial data sets} is defined as $(M, g, h^{M})$, where $(M,g)$ is an $(n+1)$-dimensional Riemannian manifold and $h^{M}$ is the second fundamental form of $M$ in $\mathsf{L}$. More precisely, the second fundamental form $h^{M}$ is defined by 
\begin{align*}
h^{M}(X_{M}, Y_{M}) = \bar{g}(\nabla^{\mathsf{L}}_{X_{M}}N_{M},Y_{M}), 
\end{align*}
where $X_{M},Y_{M} \in \mathfrak{X}(M)$ and $\nabla^{\mathsf{L}}$ is the Levi-Civita connection with respect to $\bar{g}$. By the Gauss and Codazzi equations, the initial data set $(M, g, h^{M})$ satisfies the Einstein constraint equations
\begin{align*}
&R_{g} + \left({\rm tr}_{g}h^{M}\right)^{2} - \vert h^{M} \vert^{2} = 2\mu \quad {\rm and} \\
&{\rm div}_{g}\left(h^{M} - tr_{g}h^{M}g\right) = J,
\end{align*}
where $\mu = T(N_{M},N_{M})$ and $J = T(N_{M},\cdot)$ are local \textit{energy density} and local \textit{current density}. We say that the initial data set $(M, g, h^{M})$ satisfies the \textit{dominant energy condition} if
\begin{align*}
\mu \geq \vert J \vert \,\,\, {\rm on} \,\,\, M.
\end{align*}

Let $\Sigma$ be an $n$-dimensional hypersurface in $M$ and let $N$ denote the unit normal vector field of $\Sigma$ in $M$. By convention, we choose $N$ to be the outward pointing unit normal, so that $-N$ is the inward pointing. Then $\Sigma$ admits two future-directed null normal vector fields given by 
\begin{align*}
\ell_{\pm} = N_{M} \pm N
\end{align*}
$N_{M}$ is the future-directed timelike unit normal to $M$ in $\mathsf{L}$. Then we define the \textit{null second fundamental form} $\chi^{\pm}$ by
\begin{align*}
\chi^{\pm}(X,Y) = \bar{g}(\nabla^{\mathsf{L}}_{X}\ell_{\pm}, Y) = h^{M}\vert_{\Sigma} \pm h^{\Sigma},
\end{align*}
where $X,Y \in \mathfrak{X}(\Sigma)$, $\nabla^{\mathsf{L}}$ is the Levi-Civita connection of $\bar{g}$, and $h^{\Sigma}(X,Y) = \langle \nabla^{M}_{X}N, Y \rangle$ is the second fundamental form of $\Sigma$ in $M$. The corresponding \textit{null mean curvature} is defined as
\begin{align*}
\Theta^{\pm} = {\rm tr}_{\Sigma}\chi^{\pm} = {\rm tr}_{\Sigma}h^{M} \pm H,
\end{align*}
where $H$ is the mean curvature of $\Sigma$. In particular, if the initial data set $(M, g, h^{M})$ is the time-symmetric ($h^{M} = 0$), then $\Theta^{+}(t)$ reduces to the mean curvature of $\Sigma$ in $M$.

Now we introduce the notion of MOTS. A hypersurface $\Sigma$ is called a \textit{trap surface} if both $\Theta^{+}$ and $\Theta^{-}$ are negative everywhere on $\Sigma$. It is called an  \textit{outer trap surface} if $\Theta^{+} < 0$ everywhere. A hypersurface $\Sigma$ is said to be a \textit{MOTS} if $\Theta^{+} = 0$. Finally, $\Sigma$ is said to be \textit{weakly outer trap} if $\Theta^{+} \leq 0$.

We now introduce the concept of capillary boundary. A hypersurface $\Sigma$ is said to have a \textit{capillary boundary} if it meets the boundary $\partial M$ of the Riemannian manifold $M$ at a constant angle $\theta \in (0,\pi)$. Assume that $\Sigma$ separates $M$ into two components. We label one of these components as $\Omega$, and we choose the unit normal vector $N$ to $\Sigma$ point into $\Omega$. Let $\bar{\nu}$ be the unit normal of $\partial\Sigma$, pointing out from $\Omega$ as a subset of $\partial M$. Furthermore, let $\nu$ be the outward unit normal vector of $\partial\Sigma$ in $\Sigma$ and let $\bar{N}$ be the outward unit normal of $\partial M$ in $M$. More precisely, we can choose orthonormal bases $\{N, \nu \}$ and $\{\bar{N}, \bar{\nu} \}$ such that the following relations are satisfied \cite{RS}:
\begin{align*}
&\bar{N} = (\cos\theta) N + (\sin\theta)\nu, \\
&\bar{\nu} = -(\sin\theta)N + (\cos\theta)\nu.
\end{align*}
Equivalently,
\begin{align*}
&N = -(\sin\theta)\bar{\nu} + (\cos\theta)\bar{N}, \\
&\nu = (\cos\theta)\bar{\nu} + (\sin\theta)\bar{N},
\end{align*}
where $\theta$ is the contact angle between $\Sigma$ and $\partial M$. For a given angle $\theta \in (0,\pi)$, we define the \textit{energy functional} $E: (-\epsilon,\epsilon) \rightarrow \mathbb{R}$ by
\begin{align*}
E(t) = A(t) - (\cos\theta)W(t),
\end{align*}
where $A(t)$ is the \textit{area functional} and $W(t)$ is the \textit{wetting functional}. Then the first variation formula of the energy functional is given by the following formula (cf. \cite{LON, RS}):
\begin{align} \label{first var of ener}
E'(0) = \int_{\Sigma}Hf \, dv+ \int_{\partial\Sigma} \langle X, \nu - (\cos\theta)\bar{\nu} \rangle \, ds,
\end{align}
where $f = \langle X,N \rangle$, $dv$ is the volume element on $\Sigma$, and $X$ is the variational vector field (cf. Appendix in \cite{RS}). We say that $\Sigma$ is an \textit{energy minimizing} if its energy is less than or equal to that of any nearby surface.

We now introduce the notion of stability for MOTS with capillary boundary. For closed MOTS, a notion of stability was introduced by Andersson, Mars, and Simon \cite{AMS2, AMS}. This concept was later extended to MOTS with capillary boundary by Alaee, Lesourd, and Yau in \cite{ALY}, which we now recall.

Let $\Sigma$ be an $n$-dimensional MOTS with capillary boundary in an $(n+1)$-dimensional initial data set $(M, g, h^{M})$ with boundary $\partial M$. Consider a smooth embedding map $\Phi: \Sigma \times [0,\epsilon) \rightarrow M$, $\epsilon > 0$, referred to as a $\textit{variation}$. This defines a one-parameter family $\Phi(\Sigma, t) = \Phi_{t}(\Sigma) = \Sigma_{t} \subset M$, satisfying $\Sigma_{0} = \Sigma$. The associated variation vector field is given by $X_{t} = \frac{\partial\Phi}{\partial t}$, with $X = X_{0} = \frac{\partial\Phi}{\partial t}\vert_{t=0} = \varphi N$, where $\varphi \in C^{\infty}(\Sigma)$ and $N$ is the unit normal vector field to $\Sigma$. Following \cite{ALY}, we define the functional:
\begin{align*}
F[\Sigma_{t}] = \int_{\Sigma_{t}}\Theta^{+}(t)\langle X_{t}, N_{t} \rangle \, dv + \int_{\partial\Sigma_{t}}\langle X_{t}, \nu_{t} - \cos\theta\bar{\nu}_{t} \rangle \, ds,
\end{align*}
where $\Theta^{+}(t)$ denotes the null mean curvature of $\Sigma_{t}$. Note that this functional closely resembles the first variation of the energy functional given in (\ref{first var of ener}). A computation, as in \cite{ALY, AMS}, yields the following formula for the first variation of $\Theta^{+}(t)$:
\begin{align} \label{first var null mean}
\frac{d\Theta^{+}(t)}{dt}\Bigr|_{t = 0} = -\Delta_{\Sigma}\varphi + 2\langle W,\nabla^{\Sigma}\varphi \rangle + \left({\rm div}_{\Sigma}W - \vert W \vert^{2} + Q \right)\varphi
\end{align}
and the variation of the functional $F[\Sigma_{t}]$ becomes:
\begin{align} \label{first var func}
\frac{d}{dt}\Big{|}_{t=0} F[\Sigma_{t}] = &\int_{\Sigma}\left(\vert \nabla^{\Sigma}\varphi \vert^{2} + 2\varphi\langle W, \nabla^{\Sigma}\varphi\rangle + \left({\rm div}_{\Sigma}W - \vert W \vert^{2} + Q\right)\varphi^{2}\right) dv \\
&-\int_{\partial\Sigma} q\varphi^{2} \, ds, \nonumber
\end{align}
where $Q = \frac{1}{2}R^{\Sigma} - \left(\mu + J(N)\right) - \frac{1}{2}\vert \chi_{+} \vert^{2}$ and $q = -(\cot\theta) h^{\Sigma}(\nu,\nu) + \frac{1}{\sin\theta}h^{\partial M}(\bar{\nu},\bar{\nu})$ for all $\varphi \in C^{\infty}(\Sigma)$. Additionally, the vector field $W$ on $\Sigma$ is the tangent vector field dual to the 1-form $h^{M}(N,\cdot)$. Then the corresponding eigenvalue problem associated with the functional (\ref{first var func}) is given by:
\begin{align} \label{stable ope}
\begin{cases}
&L\varphi := -\Delta_{\Sigma}\varphi + 2\langle W,\nabla^{\Sigma}\varphi \rangle + \left({\rm div}_{\Sigma}W - \vert W \vert^{2} + Q \right)\varphi = \lambda \varphi \quad {\rm on} \,\, \Sigma, \\
&B\varphi := \frac{\partial\varphi}{\partial\nu} - q\varphi = 0 \quad {\rm along} \,\, \partial\Sigma.
\end{cases}
\end{align}
The operator $\mathcal{L} := (L,B)$ is called \textit{stability operator} with Robin-type boundary condition. A compact, MOTS with capillary boundary $\Sigma$ in an initial data set $(M,g,h^{M})$ is said to be \textit{stable} if there exists a non-negative function $\varphi \in C^{\infty}(\Sigma)$, $\varphi \neq 0$, satisfying the boundary condition $B\varphi = 0$ such that $L\varphi \geq 0$. By the maximum principle, we may assume without loss of generality that $\varphi > 0$ on $\Sigma$.

We now present the boundary conditions for an initial data set $(M, g, h^{M})$ when MOTS with capillary boundary is embedded in an initial data set with boundary. First, the second fundamental form of $\partial M$ in $(M,g)$ is defined by
\begin{align*}
h^{\partial M}(X,Y) = \langle \nabla^{M}_{X_{\partial M}} \bar{N},Y_{\partial M} \rangle,
\end{align*}
where $X_{\partial M},Y_{\partial M} \in \mathfrak{X}(\partial M)$ and $\bar{N}$ is the unit normal of $\partial M$ in $(M,g)$. We say that the initial data set $(M, g, h^{M})$ satisfies the \textit{tilted dominant boundary energy condition} if the following inequality holds on $\partial M$:
\begin{align*}
H^{\partial M} + (\cos\theta){\rm tr}_{\partial M}h^{M} \geq \sin\theta\vert h^{M}(\bar{N}, \cdot)^{T} \vert,
\end{align*}
where $h^{M}(\bar{N},\cdot)^{T}$ is the component of the 1-form $h^{M}(\bar{N},\cdot)$ tangential to $\partial M$. This condition was introduced by Chai \cite{CH, CW}. In the special case $\theta = \frac{\pi}{2}$, it coincides with the dominant boundary condition introduced in \cite{ALM}.

Finally, we introduce the definition of weakly outermost and outermost. Let $\Sigma$ be a MOTS that separates the initial data set $(M, g, h^{M})$. Denote by $M_{+}$ the region consisting of $\Sigma$ and the region outside of $\Sigma$. We say that $\Sigma$ is \textit{weakly outermost} if there exists no outer trap surface $(\Theta^{+} < 0)$ in $M_{+}$ that is homologous to $\Sigma$. Similarly, $\Sigma$ is said to be \textit{outermost} if there exists no weakly outer trap surface $(\Theta^{+} \leq 0)$ in $M_{+}$ that is homologous to $\Sigma$.

\section{Auxiliary results}

In general, the operator $L$ defined in (\ref{stable ope}) of Section 2 is not self-adjoint. To address this, we consider the ``symmetrized'' operator $L_{s}$. For closed MOTS, This operator was first observed by Galloway and Schoen \cite{GA, GS}. The following lemma extends the results associated with the symmetrized operator from the closed case to the setting of MOTS with capillary boundary.

\begin{lemma} \label{lem 3.1}
Let $(M,g,h^{M})$ be an $(n+1)$-dimensional, initial data set with boundary $\partial M$ and $\Sigma$ be a compact, MOTS with capillary boundary in $(M,g,h^{M})$. If $\Sigma^{n}$ is a stable MOTS, then the first eigenvalue $\lambda_{1}(L_{s})$ of $L_{s}:= -\Delta_{\Sigma} + Q$ on $\Sigma$ with Robin-type boundary condition $B_{s}\varphi := \frac{\partial\varphi}{\partial\nu} - (q - \langle W,\nu \rangle)\varphi = 0$ is non-negative.
\end{lemma}

\begin{proof}
Let $\varphi \in C^{\infty}(\Sigma)$ be a positive function satisfying the Robin-type boundary condition $B\varphi = 0$ such that $L\varphi \geq 0$. Then
\begin{align} \label{eq 3.1}
\frac{L\varphi}{\varphi} &= -\frac{\Delta_{\Sigma}\varphi}{\varphi} + \frac{2}{\varphi}\langle W, \nabla^{\Sigma}\varphi \rangle + \left({\rm div}_{\Sigma}W - \vert W \vert^{2} + Q\right) \\
&= - \frac{\Delta_{\Sigma}\varphi}{\varphi} + {\rm div}_{\Sigma}W + Q - \vert W - \nabla^{\Sigma}{\rm ln} \, \varphi \vert^{2} + \vert \nabla^{\Sigma}{\rm ln} \, \varphi \vert^{2} \nonumber \\
&= {\rm div}_{\Sigma}Y + Q - \vert Y \vert^{2}, \nonumber
\end{align}
where $Y = W - \nabla^{\Sigma}{\rm ln} \, \varphi$. Multiplying the equation (\ref{eq 3.1}) by $\phi^{2}$, then we have
\begin{align*}
\phi^{2}\frac{L\varphi}{\varphi} &= \phi^{2}{\rm div}_{\Sigma}Y + \phi^{2} Q - \phi^{2}\vert Y \vert^{2} \\
&= {\rm div}_{\Sigma}(\phi^{2}Y) - 2\phi \langle \nabla^{\Sigma}\phi, Y \rangle + Q\phi^{2} - \vert QY \vert^{2} \\
&= {\rm div}_{\Sigma}(\phi^{2}Y) + Q\phi^{2} - \vert QY + \nabla^{\Sigma}\phi \vert^{2} + \vert \nabla^{\Sigma}\phi \vert^{2} \\
&\leq {\rm div}_{\Sigma}(\phi^{2} Y) + \vert \nabla^{\Sigma}\phi \vert^{2} + Q\phi^{2}
\end{align*}
for any smooth function $\phi \in C^{\infty}(\Sigma)$. Let $\phi$ as a first eigenfunction of $L_{s}$, $L_{s}\phi = \lambda_{1}(L_{s})\phi$, with $B_{s}\phi = 0$. Using $L\varphi \geq 0$, we obtain
\begin{align*}
0 \leq& \int_{\Sigma}\left(\vert \nabla^{\Sigma}\phi \vert^{2} + Q\phi^{2}\right)dv + \int_{\partial\Sigma}\phi^{2}\left(\langle W,\nu \rangle - \frac{1}{\varphi}\frac{\partial\varphi}{\partial\nu}\right)ds \\
=& \int_{\Sigma}\left(\vert \nabla^{\Sigma}\phi \vert^{2} + Q\phi^{2}\right)dv + \int_{\partial\Sigma}\phi^{2}\left(\langle W,\nu \rangle - q\right)ds \\
=& \int_{\Sigma} \left(\vert \nabla^{\Sigma} \phi \vert^{2} + Q\phi^{2}\right)dv - \int_{\partial\Sigma}\frac{\partial\phi}{\partial\nu}\phi \, ds, \\
=& \lambda_{1}(L_{s})\int_{\Sigma}\phi^{2} dv,
\end{align*}
where we used $B\varphi = 0$ and $\phi B_{s}\phi = 0$. This implies that $\lambda_{1}(L_{s}) \geq 0$.

\end{proof}

 If we consider the Rayleigh formula, then the first eigenvalue of the operator $L_{s}$ as follows:
\begin{align} \label{sym eigen}
\lambda_{1}(L_{s}) = \inf_{\varphi \in C^{\infty}\backslash \{0\}} \frac{\int_{\Sigma}\left(\vert \nabla^{\Sigma}\varphi \vert^{2} + Q\varphi^{2}\right)dv - \int_{\partial\Sigma} \left(q - \langle W, \nu \rangle\right)\varphi^{2} ds}{\int_{\Sigma}\varphi^{2}dv},
\end{align}
where $Q = \frac{1}{2}R^{\Sigma} - \left(\mu + J(N)\right) - \frac{1}{2}\vert \chi_{+} \vert^{2}$ and $q = -(\cot\theta) h^{\Sigma}(\nu,\nu) + \frac{1}{\sin\theta}h^{\partial M}(\bar{\nu},\bar{\nu})$.

The following lemma simplifies $q$.

\begin{lemma} \label{lem 3.2}
Let $(M, g, h^{M})$ be an $(n+1)$-dimensional initial data set with boundary $\partial M$ and $\Sigma^{n}$ be a hypersurface with capillary boundary in $M$ with contact angle $\theta \in (0,\pi)$. Then
\begin{align} \label{eq 3.2}
h^{\partial M}(\bar{\nu},\bar{\nu}) - (\cos\theta)h^{\Sigma}(\nu,\nu) + (\sin\theta)H^{\partial\Sigma} = H^{\partial M} - (\cos\theta)H.  
\end{align}
In particular, if $n = 2$, then $H^{\partial\Sigma} = k^{\partial\Sigma}$, where $k^{\partial\Sigma}$ is the geodesic curvature of $\partial\Sigma$ in $\Sigma$.
\end{lemma}

\begin{proof}
Let $\{T_{1}, \cdots, T_{n-1}\}$ be a local orthonormal frame for $\partial\Sigma$. Since $h^{\Sigma}(\nu,\nu) = H - \sum^{n-1}_{i=1}h^{\Sigma}(T_{i},T_{i})$, we have
\begin{align*}
-(\cos\theta)h^{\Sigma}(\nu,\nu) + (\cos\theta)H + (\sin\theta)H^{\partial\Sigma} &= (\cos\theta)\sum^{n-1}_{i=1}h^{\Sigma}(T_{i},T_{i}) + (\sin\theta)H^{\partial\Sigma} \\
&= (\cos\theta)\sum^{n-1}_{i=1}\langle \nabla^{M}_{T_{i}}N, T_{i} \rangle + (\sin\theta)\sum^{n-1}_{i=1}\langle \nabla^{M}_{T_{i}} \nu, T_{i} \rangle \\
&= \sum^{n-1}_{i=1}h^{\partial M}(T_{i},T_{i}).
\end{align*}
Thus we obtain
\begin{align*}
-(\cos\theta)h^{\Sigma}(\nu,\nu) + (\cos\theta)H + (\sin\theta)H^{\partial\Sigma} + h^{\partial M}(\bar{\nu},\bar{\nu}) &= \sum^{n-1}_{i=1}h^{\partial M}(T_{i},T_{i}) + h^{\partial M}(\bar{\nu},\bar{\nu}) \\
&= H^{\partial M}.
\end{align*}
\end{proof}
Considering Lemma \ref{lem 3.2}, we write $q$ as follows
\begin{align} \label{simple q}
q = -(\cot\theta)h^{\Sigma}(\nu,\nu) + \frac{1}{\sin\theta}h^{\partial M}(\bar{\nu},\bar{\nu}) = \frac{1}{\sin\theta}\left(H^{\partial M} - (\cos\theta)H - (\sin\theta)H^{\partial\Sigma}\right).    
\end{align}

\section{3-dimensional initial data set}

In this section, we establish area estimates for stable MOTS with capillary boundary and prove a rigidity result for 3-dimensional initial data sets. We begin by presenting the area estimates.

\begin{theorem} [Theorem \ref{main thm 1}] \label{main thm 1-1}
Let $(M, g, h^{M})$ be a 3-dimensional initial data set with boundary and $\Sigma^{2}$ be a compact stable MOTS with capillary boundary in $(M, g, h^{M})$. Assume that $\mu + J(N) \geq C$ for $C \in \mathbb{R}$ and $(M, g, h^{M})$ satisfies the tilted dominant boundary energy condition. Then
\begin{align} \label{area est 1-1}
A(\Sigma)C \leq 2\pi\chi(\Sigma),
\end{align}
where $\chi(\Sigma)$ is the Euler characteristic. If equality holds in (\ref{area est 1-1}), then we have
\begin{itemize}
\item[(1)] $\mu + J(N) = C$, $\chi^{+} = 0$, and the Gaussian curvature $K^{\Sigma} = C$ on $\Sigma$,
\item[(2)] $\lambda_{1}(L_{s}) = \lambda_{1}(L) = 0$ on $\Sigma$,
\item[(3)] $H^{\partial M} + (\cos\theta) {\rm tr}_{\partial M}h^{M} = \sin\theta\vert h^{M}(\bar{N},\cdot)^{T} \vert$ and the geodesic curvature $k^{\partial\Sigma} = 0$ along $\partial\Sigma$ in $\Sigma$.
\end{itemize}
\end{theorem}

\begin{proof}
Since $\Sigma$ is a stable MOTS, it follow that $\lambda_{1}(L_{s}) \geq 0$ from the previous Lemma \ref{lem 3.1}. Then we have
\begin{align*}
0 \leq \inf_{\varphi \in C^{\infty}\backslash \{0\}}\frac{\int_{\Sigma}\left(\vert \nabla^{\Sigma}\varphi \vert^{2} + Q\varphi^{2}\right)dv - \int_{\partial\Sigma} \left(q - \langle W, \nu \rangle\right)\varphi^{2} ds}{\int_{\Sigma}\varphi^{2}dv}.
\end{align*}
If we take $\varphi = 1$, then we obtain
\begin{align*}
0 \leq \int_{\Sigma}Q \, dv - \int_{\partial\Sigma}\left(q - \langle W,\nu \rangle\right)ds.
\end{align*}
By the definition of $Q$ and the equation (\ref{simple q}), we get
\begin{align*}
0 \leq& \int_{\Sigma}K^{\Sigma} \, dv - \int_{\Sigma}\left(\mu + J(N)\right) dv - \frac{1}{2}\int_{\Sigma}\vert \chi^{+} \vert^{2}dv \\
&+\int_{\partial\Sigma}\left(-\frac{1}{\sin\theta}H^{\partial M} + \frac{\cos\theta}{\sin\theta}H + k^{\partial\Sigma} + \langle W,\nu \rangle\right)ds \\
\leq& \int_{\Sigma}K^{\Sigma} \, dv - \int_{\Sigma}\left(\mu + J(N)\right) dv + \int_{\partial\Sigma}k^{\partial\Sigma} ds \\
&+\int_{\partial\Sigma}-\frac{1}{\sin\theta}\left(H^{\partial M} + (\cos\theta){\rm tr}_{\Sigma}h^{M} - (\sin\theta)h^{M}(N,\nu)\right)ds,
\end{align*}
where $K^{\Sigma}$ is the Gaussian curvature of $\Sigma$ and $k^{\partial\Sigma}$ is the geodesic curvature of $\partial\Sigma$ in $\Sigma$. Now we claim that 
\begin{align} \label{tilted}
(\cos\theta){\rm tr}_{\Sigma}h^{M} - (\sin\theta)h^{M}(N,\nu) = (\cos\theta){\rm tr}_{\partial M}h^{M} + (\sin\theta)h^{M}(\bar{N},\bar{\nu}).
\end{align}Let $\{e_{1}, \nu\} \in T{\Sigma}$. Then we have
\begin{align*}
(\cos\theta){\rm tr}_{\Sigma}h^{M} - (\sin\theta)h^{M}(N,\nu) &= (\cos\theta)h^{M}(e_{1},e_{1}) + (\cos\theta)h^{M}(\nu,\nu) - (\sin\theta)h^{M}(N,\nu) \\
&= (\cos\theta)h^{M}(e_{1},e_{1}) + h^{M}((\cos\theta)\nu - (\sin\theta)N,\nu) \\
&= (\cos\theta)h^{M}(e_{1},e_{1}) + h^{M}(\bar{\nu},\nu) \\
&= (\cos\theta)\left(h^{M}(e_{1},e_{1}) + h^{M}(\bar{\nu},\bar{\nu})\right) + (\sin\theta)h^{M}(\bar{N},\bar{\nu}) \\
&= (\cos\theta){\rm tr}_{\partial M}h^{M} + (\sin\theta)h^{M}(\bar{N},\bar{\nu})
\end{align*}
(see Appendix in \cite{CH}). So we prove our claim. By the equation (\ref{tilted}) and the Gauss-Bonnet theorem, we obtain
\begin{align*}
0 &\leq 2\pi\chi(\Sigma) - \int_{\Sigma}\left(\mu + J(N)\right)dv \\
&\quad -\int_{\Sigma}\frac{1}{\sin\theta}\left(H^{\partial M} + (\cos\theta){\rm tr}_{\partial M}h^{M} + (\sin\theta)h^{M}(\bar{N},\bar{\nu})\right)dv.
\end{align*}
By the our dominant energy condition, we have
\begin{align*}
0 \leq 2\pi\chi(\Sigma) - A(\Sigma)C.
\end{align*}
This implies that
\begin{align} \label{eq 4.1}
A(\Sigma)C \leq 2\pi\chi(\Sigma).
\end{align}
If the equality holds in (\ref{eq 4.1}), then all inequalities above become equalities. So we deduce that $\mu + J(N) = C$, $\chi^{+} = 0$, and $Q = K^{\Sigma} - C$ on $\Sigma$. Moreover, $H^{\partial M} + (\cos\theta){\rm tr}_{\partial M}h^{M} = (\sin\theta)\vert h^{M}(\bar{N},\cdot)^{T} \vert = -(\sin\theta)h^{M}(\bar{N},\bar{\nu})$ along $\partial\Sigma$. This implies that $q - \langle W,\nu \rangle = -k^{\partial\Sigma}$. On the other hand, by Lemma \ref{lem 3.1}, we know that $\lambda_{1}(L_{s}) \geq 0$. But, if we take $\varphi = 1$, then the right hand side of (\ref{sym eigen}) is zero, which implies that $\lambda_{1}(L_{s}) = 0$ and $\varphi = 1$ is an associated eigenfunction. More precisely, $\varphi = 1$ is a solution to
\begin{align*}
\begin{cases}
-\Delta_{\Sigma}\varphi + \left(K^{\Sigma} - C\right)\varphi = 0 \\
\frac{\partial \varphi}{\partial\nu} + k^{\partial\Sigma}\varphi = 0.
\end{cases}
\end{align*}
Hence, we deduce that $K^{\Sigma} = C$ on $\Sigma$ and $k^{\partial\Sigma} = 0$ along $\partial\Sigma$.

\end{proof}

We now prove a lemma that will be used in the construction of the foliation. The following Lemma is analogous to Lemma 3.5 in \cite{M2}.

\begin{lemma} \label{lem 4.1}
Under the assumption of the Theorem \ref{main thm 1-1} and $A(\Sigma)C = 2\pi\chi(\Sigma)$, zero is a simple eigenvalue of $L$ on $\Sigma$ with Robin-type boundary condition $B\varphi = 0$ whose associated eigenfunction can be chosen to be positive. The same holds for the formal adjoint operator $L^{*}$ on $\Sigma$ with Robin-type boundary condition $B^{*}\varphi^{*} = 0$, where 
\begin{align*}
\begin{cases}
L^{*}\varphi = -\Delta_{\Sigma}\varphi - 2 \langle W,\nabla^{\Sigma}\varphi \rangle + \varphi\left(Q - \vert W \vert^{2} - {\rm div}_{\Sigma}W\right)\\
B^{*}\varphi = \frac{\partial\varphi}{\partial\nu} - \varphi\left(q - 2\langle W,\nu \rangle\right).
\end{cases}
\end{align*}
\end{lemma}

\begin{proof}
By the Theorem \ref{main thm 1}, we know that $Q = 0$ on $\Sigma$ and $q - \langle W,\nu \rangle = 0$ along $\partial\Sigma$. Since $\Sigma$ is a stable MOTS, we take $0 < \varphi \in C^{\infty}(\Sigma)$. Then $\varphi$ is satisfied the following
\begin{align*}
\begin{cases}
L\varphi \geq 0 \,\, {\rm on} \,\, \Sigma, \\
B\varphi = 0 \,\, {\rm along} \,\, \partial\Sigma.
\end{cases}
\end{align*}
Then we have
\begin{align*}
0 \leq \frac{L\varphi}{\varphi} = - \vert W - \nabla^{\Sigma}{\rm ln}\varphi \vert^{2} + {\rm div}_{\Sigma}\left(W - \nabla^{\Sigma}{\rm ln}\varphi\right).
\end{align*}
So we get
\begin{align*}
\int_{\Sigma}\vert W - \nabla^{\Sigma}{\rm ln}\varphi \vert^{2}dv &\leq \int_{\Sigma}{\rm div}_{\Sigma}\left(W - \nabla^{\Sigma}{\rm ln}\varphi\right)dv \\
&= \int_{\partial\Sigma}\left(\langle W,\nu \rangle - \frac{1}{\varphi}\frac{\partial\varphi}{\partial\nu}\right)ds \\
&= \int_{\partial\Sigma}\left(\langle W,\nu \rangle - q\right)ds \\
&= 0,
\end{align*}
where the last equality uses item (3) of Theorem \ref{main thm 1}. Hence, we have $W = \nabla^{\Sigma}{\rm ln}\varphi$. This follows that $L\varphi = 0$. Therefore, we deduce that $\lambda = 0$ is an eigenvalue of $L$ with boundary condition $B\varphi = 0$. Now we prove that $\lambda = 0$ is a simple. Let $\phi \in C^{\infty}(\Sigma)$ be such that $L\phi = 0$ on $\Sigma$ and $B\phi = 0$ along $\partial\Sigma$. Define $\psi = \frac{\phi}{\varphi}$. By the simple calculation, we get
\begin{align*}
\Delta_{\Sigma}\psi &= \frac{\Delta_{\Sigma}\phi}{\varphi} - 2\frac{\langle \nabla^{\Sigma}\phi, \nabla^{\Sigma}\varphi \rangle}{\varphi^{2}} - \frac{\Delta_{\Sigma}\varphi\cdot\phi}{\varphi^{2}} + 2\frac{\vert \nabla^{\Sigma}\varphi \vert^{2}}{\varphi^{3}}Q \\
&= - \frac{1}{\varphi}L\phi \\
&= 0 \,\, {\rm on} \,\, \Sigma.
\end{align*}
and
\begin{align*}
\frac{\partial\psi}{\partial\nu} = \frac{1}{\varphi}\left(\frac{\partial\phi}{\partial\nu} - \frac{\phi}{\varphi}\frac{\partial\varphi}{\partial\nu}\right) = 0 \,\, {\rm along} \,\, \partial\Sigma.
\end{align*}
Hence, $\psi = \frac{\phi}{\varphi}$ is a constant. This implies that $\lambda = 0$ is a simple eigenvalue whose associated eigenfunction are given by $\phi = c\varphi$ for $c \in \mathbb{R}$. Next, we show that $\lambda^{*} = 0$ is a simple eigenvalue of $L^{*}$ on $\Sigma$ with Robin-type boundary condition $B^{*}\varphi^{*} = 0$. Let $\varphi^{*} \in C^{\infty}(\Sigma)$. Then we have
\begin{align*}
L^{*}\varphi^{*} &= - \Delta_{\Sigma}\varphi^{*} - 2\langle \frac{\nabla^{\Sigma}\varphi}{\varphi}, \nabla^{\Sigma}\varphi^{*} \rangle - \varphi^{*}\left(\frac{\vert \nabla^{\Sigma}\varphi \vert^{2}}{\varphi^{2}} + {\rm div}_{\Sigma}\left(\frac{\nabla^{\Sigma}\varphi}{\varphi}\right)\right) \\
&= \varphi\Delta_{\Sigma}\varphi^{*} + 2\langle \nabla^{\Sigma}\varphi^{*}, \nabla^{\Sigma}\varphi \rangle + \varphi^{*}\Delta_{\Sigma}\varphi \\
&= -\frac{\Delta_{\Sigma}\left(\varphi^{*}\varphi\right)}{\varphi}
\end{align*}
and
\begin{align*}
B^{*}\varphi^{*} &= \frac{\partial\phi^{*}}{\partial\nu} - \varphi^{*}\left( q - 2\langle W, \nu \rangle\right) \\
&= \frac{\partial\varphi^{*}}{\partial\nu} + \varphi^{*}\langle W,\nu \rangle \\
&= \frac{\partial\varphi^{*}}{\partial\nu} + q\varphi^{*} \\
&= \frac{1}{\varphi}\frac{\partial}{\partial\nu}\left(\varphi^{*}\varphi\right).
\end{align*}
It follows that $L^{*}\varphi^{*} = 0$ on $\Sigma$ with Robin-type boundary condition $B^{*}\varphi^{*} = 0$ if and only if $\varphi^{*}\varphi$ is a constant. Hence, $\lambda^{*} = 0$ is a simple eigenvalue of $L^{*}$ with $B^{*}\varphi^{*} = 0$ whose associated eigenfunction are given by $\varphi^{*} = \frac{c}{\varphi}$ for $c\in \mathbb{R}$.

\end{proof}

The following proposition shows that it is possible to construct a foliation by hypersurfaces with constant null mean curvature near a MOTS with capillary boundary in a 3-dimensional initial data set, under suitable conditions.

\begin{proposition} \label{foliation}
Under the assumptions of Theorem \ref{main thm 1-1} and $A(\Sigma)C = 2\pi\chi(\Sigma)$, there exist a neighborhood $V \cong (-\epsilon,\epsilon) \times \Sigma$ of $\Sigma = \{0\} \times \Sigma$ in $M^{3}$ and a positive function $\varphi_{positive}: V \rightarrow \mathbb{R}$ such that
\begin{itemize}
\item[(a)] $g|_{V} = \varphi^{2}_{positive}dt^{2} + \gamma_{t}$, where $\gamma_{t}$ is the induced metric on $\Sigma_{t} \cong \{t\} \times \Sigma$.
\item[(b)] Each $\Sigma_{t}$ is a capillary boundary hypersurface in $(M^{3}, g, h^{M})$ with contact angle equal to $\theta$ and constant null mean curvature $\Theta^{+}(t)$ with respect to the outward unit normal $N_{t} = \varphi^{-1}_{positive}\frac{\partial}{\partial t}$, where $N_{0} = N$.
\item[(c)] $\frac{\partial\varphi_{positive}}{\partial\nu_{t}} = q\varphi_{positive}$ along $\partial\Sigma_{t}$, where $\nu_{t}$ is the outward unit normal of $\partial\Sigma_{t}$ in $(\Sigma_{t}, \gamma_{t})$. 
\end{itemize}
\end{proposition}

\begin{proof}
Let $Z$ be a smooth vector field on $M^{3}$ such that $Z_{x} = N_{x}$ for $x \in \Sigma$ and $Z_{p} \in T_{p}\partial M$ for $p \in \partial M$. We denote by $\Psi = \Psi(p,t)$ the flow of $Z$. Since $\Sigma$ is a compact, there exists $\delta > 0$ such that $\Psi(x,t)$ is well-defined for $(x,t) \in \Sigma \times (-\delta,\delta)$. 

 Fix $0 < \alpha < 1$ and let $B_{\delta}(0) = \{\mathsf{b} \in C^{2,\alpha} \,\, \vert \,\, \Vert \mathsf{b}\Vert_{2,\alpha} < \delta\}$. We consider $\Sigma_{\mathsf{b}} = \{\Psi(x,\mathsf{b}(s)) \,\, \vert \,\, x \in \Sigma\}$ and $\Theta^{+}_{\mathsf{b}}$ is the null mean curvature of $\Sigma_{b}$ with respect to the outward unit normal $N_{\mathsf{\mathsf{b}}}$ of $\Sigma_{\mathsf{b}}$.

 We define the function $\Xi: B_{\delta}(0) \times \mathbb{R} \rightarrow C^{0,\alpha}(\Sigma) \times C^{1,\alpha}(\partial\Sigma) \times \mathbb{R}$ by
\begin{align*}
\Xi(\mathsf{b},\mathsf{k}) = \left(\Theta^{+}_{\mathsf{b}} - \mathsf{k}, \cos\theta - \langle N_{\mathsf{b}}, \bar{N}_{\mathsf{b}} \rangle, \int_{\Sigma}\mathsf{b} \, dv\right).
\end{align*}
Then we have (see Lemma A.2 in \cite{LI} and Lemma 3.3 in \cite{LON})
\begin{align*}
D\Xi_{(0,0)}(\mathsf{b},\mathsf{k}) = \frac{d}{ds}\Bigr|_{s = 0} \Xi(s\mathsf{b},s\mathsf{k}) = \left(L\mathsf{b} - \mathsf{k}, (\sin\theta)B\mathsf{b}, \int_{\Sigma}\mathsf{b} \, dv\right).
\end{align*}
First, we claim that $D\Xi_{(0,0)}$ is an isomorphism. If $\left(\mathsf{b},\mathsf{k}\right) \in {\rm ker}\left(D\Xi_{(0,0)}\right)$, then we obtain
\begin{align} \label{eq 4.2}
\begin{cases}
L\mathsf{b} = \mathsf{k} \,\, {\rm on} \,\, \Sigma, \\
B\mathsf{b} = 0 \,\, {\rm along} \,\, \partial\Sigma, \\
\int_{\Sigma} \mathsf{b} \, dv = 0.
\end{cases}
\end{align}
Let $\varphi^{*}$ be an eigenfunction of $L^{*}$ on $\Sigma$ with $B^{*}\varphi^{*} = 0$. Since Lemma \ref{lem 4.1}, there exists the eigenvalue $\lambda^{*} = 0$ of $L^{*}$ with $B^{*}\varphi^{*}$. Multiplying the equation (\ref{eq 4.2}) by $\varphi^{*}$ and integrating over $\Sigma$, we get
\begin{align*}
\mathsf{k}\int_{\Sigma}\varphi^{*} \, dv = \int_{\Sigma}\varphi^{*}L\mathsf{b} \, dv = \int_{\Sigma}\mathsf{b}L^{*}\varphi^{*} dv + \int_{\partial\Sigma}\mathsf{b}B^{*}\varphi^{*} ds - \int_{\partial\Sigma}\varphi^{*}B\mathsf{b} \, ds = 0.
\end{align*}
This implies that $\mathsf{k} = 0$. So we have that $L\mathsf{b} = 0$ on $\Sigma$, $B\mathsf{b} = 0$ along $\partial\Sigma$, and $\int_{\Sigma}\mathsf{b} \, dv = 0$. By the Lemma \ref{lem 4.1}, we get $\mathsf{b} = c\varphi$, where $c \in \mathbb{R}$ and $\varphi > 0$ is an eigenfunction of $L$ with $B\varphi = 0$, associated with the eigenvalue $\lambda = 0$. Since $\int_{\Sigma}\mathsf{b} \, dv = 0$, we have $\mathsf{b} = 0$. Hence, $D\Xi_{(0,0)}$ is injective. Next, we show that $D\Xi_{(0,0)}$ is onto. By the Fredholm alternative, the following problem
\begin{align*}
\begin{cases}
L\mathsf{b} = f_{1} \,\, {\rm on} \,\, \Sigma, \\
B\mathsf{b} = f_{2} \,\, {\rm along} \,\, \partial\Sigma
\end{cases}
\end{align*}
has a solution if and only if 
\begin{align*}
\int_{\Sigma}\varphi^{*}f_{1}dv + \int_{\partial\Sigma}\varphi^{*}f_{2}ds = 0.
\end{align*}
Let $\left(f_{1},(\sin\theta)f_{2},c\right) \in C^{0,\alpha}(\Sigma) \times C^{1,\alpha}(\partial\Sigma) \times \mathbb{R}$ and
\begin{align*}
\mathsf{k}_{0} = -\frac{\int_{\Sigma}\varphi^{*}f_{1}dv + \int_{\partial\Sigma}\varphi^{*}f_{2}ds}{\int_{\Sigma}\varphi^{*}dv}.
\end{align*}
We take $\mathsf{b}_{0} \in C^{2,\alpha}(\Sigma)$ such that $L\mathsf{b}_{0} = \mathsf{k}_{0} + f_{1}$ on $\Sigma$ and $B\mathsf{b}_{0} = f_{2}$ along $\partial\Sigma$. By the Fredholm alternative, there exists
\begin{align*}
\mathsf{t}_{0} = \frac{c - \int_{\Sigma}\mathsf{b}_{0}dv}{\int_{\Sigma}\varphi \, dv}.
\end{align*}
Then we obtain
\begin{align*}
D\Xi_{(0,0)}\left(\mathsf{b}_{0} + \mathsf{t}_{0}\varphi, \mathsf{k}_{0}\right) &= \frac{d}{ds}\Bigr|_{s = 0}\Xi(s\left(\mathsf{b}_{0} + t_{0}\varphi\right), s\mathsf{k}_{0}) \\
&= \left(f_{1}, (\sin\theta)f_{2}, c\right).
\end{align*}
It follows that $D\Xi_{(0,0)}$ is onto. Hence, $D\Xi_{(0,0)}$ is an isomorphim. Now we can apply inverse function theorem that for $s \in \mathbb{R}$ sufficiently small. By the inverse function theorem, there exists $\mathsf{b}(s) \in B_{\delta}(0)$ and $\mathsf{k}(s) \in \mathbb{R}$ such that $\Xi(\mathsf{b}(s),\mathsf{k}(s)) = (0,0,s)$ with $\mathsf{b}(0) = 0$ and $\mathsf{k}(0) = 0$. By the chain rule, we obtain
\begin{align*}
\left(L\mathsf{b}'(0) - \mathsf{k}'(0), (\sin\theta)B\mathsf{b}'(0), \int_{\Sigma}\mathsf{b}'(0) \, dv\right) &= D\Xi_{(0,0)}\left(\mathsf{b}'(0),\mathsf{k}'(0)\right) \\
&= \frac{d}{ds}\Bigr|_{s=0}\Xi(\mathsf{b}(s),\mathsf{k}(s)) = (0,0,1).
\end{align*}
By the same argument as our claim, we have $\mathsf{k}'(0) = 0$ and $\mathsf{b}'(0) = c\varphi$ for $c \in \mathbb{R}$. Since $\int_{\Sigma}\mathsf{b}'(0) \, dv = 1$, $c > 0$. For sufficiently small $\epsilon > 0$, we obtain that $\{\Sigma_{\mathsf{b}(s)}\}_{\vert s \vert < \epsilon}$ forms a foliation of a neighborhood of $\Sigma$ in $M$ with capillary boundary hypersurface and constant null mean curvature.

 Finally, if we choose the coordinated $(t, x^{i})$ in a neighborhood $V \cong (-\epsilon,\epsilon) \times \Sigma$, then each slice $\Sigma_{t} \cong \{t\} \times \Sigma$ is a capillary boundary hypersurface in $M$ with constant null mean curvature $\Theta^{+}(t)$ with respect to the unit normal $N_{t}$. Moreover, these coordinates can be chosen in a such way that $\frac{\partial}{\partial t} = \varphi_{positive}N_{t}$ on $\Sigma_{t}$. Note that item (c) follows from the capillary boundary condition (see Lemma A.2 in \cite{LI}).

\end{proof}

Let $(M, g, h^{M})$ be an $(n+1)$-dimensional initial data set with boundary. We denote that $h^{M}$ is \textit{$n$-convex} if ${\rm tr}_{\pi}h^{M} \geq 0$ for all $\pi \subset T_{p}M$ and $p \in M$, where $\pi$ is a linear subspace of dimension $n$ (cf. \cite{AM, M, M2}).

We now prove a rigidity result for 3-dimensional initial data sets with boundary.

\begin{theorem} [Theorem \ref{main thm 2}] \label{main thm 1-2}
Let $(M, g ,h^{M})$ be a 3-dimensional initial data sets with boundary and $\Sigma^{2}$ be a compact stable MOTS with capillary boundary in $(M, g, h^{M})$ that is weakly outermost in $M_{+}$. Assume that $\mu - \vert J \vert \geq C$ for $C \in \mathbb{R}$ and $(M, g, h^{M})$ satisfies the tilted dominant boundary energy condition, all of which hold on $M_{+}$. Suppose that the equality holds in (\ref{area est 1-1}).
\begin{itemize}
\item[(I)] If $C > 0$ and ${\rm tr}_{\pi}h^{M} \leq 0$ with $\theta \in [\frac{\pi}{2},\pi)$, then $\theta$ must be $\frac{\pi}{2}$.
\item[(II)] If $C < 0$ and 2-convex with $\theta \in (0,\frac{\pi}{2}]$, then $\theta$ must be $\frac{\pi}{2}$.
\item[(III)] If $C = 0$, 2-convex, and an energy-minimizing, then $\theta = \frac{\pi}{2}$ or not.
\end{itemize}
Moreover, the case of $\theta = \frac{\pi}{2}$, the following hold:
\begin{itemize}
\item[(i)] there exists an outer neighborhood $V$ of $\Sigma$ in $M$ such that $(V,g\vert_{V})$ is isometric to $([0,\epsilon) \times \Sigma, dt^{2} + \gamma)$, where $\gamma = g|_{\Sigma}$ and $(\Sigma,\gamma)$ have the constant Gaussian curvature $K^{\Sigma} = C$ and zero geodesic curvature $k^{\partial\Sigma} = 0$;
\item[(ii)] $h^{M} = a dt^{2}$ on $V$, where $a \in C^{\infty}$ depends only on $t \in [0,\epsilon)$;
\item[(iii)] $\mu = C$ and $J = 0$ on $V$;
\item[(iv)] The tilted dominant boundary condition is saturated along $V \cap \partial M$.
\end{itemize}
Otherwise, if without loss of generality of {\rm (III)}, $H^{\partial M} \geq 0$, and $\theta \neq \frac{\pi}{2}$, then $\Sigma$ is flat with $k^{\partial\Sigma} = 0$ and $M$ is flat with $H^{\partial M} = 0$ in an outer neighborhood of $\Sigma$.
\end{theorem}

\begin{proof}
Our assumption is that $\Sigma$ is weakly outermost and we know that $\Theta^{+}(t)$ is the constant along $\Sigma_{t}$. So we have $\Theta^{+}(t) \geq 0$ for all $t \in [0,\epsilon)$. From (\ref{first var null mean}), we get
\begin{align} \label{eq 5.1}
\frac{d\Theta^{+}(t)}{dt} =& -\Delta_{\Sigma_{t}}\varphi + 2\langle W_{t}, \nabla^{\Sigma_{t}}\varphi \rangle + \left(Q - \vert W_{t} \vert^{2} + {\rm div}_{\Sigma_{t}}W_{t}\right)\varphi \\
&+ \left(-\frac{1}{2}\left(\Theta^{+}(t)\right)^{2} + \Theta^{+}(t){\rm tr}_{g}h^{M}\right)\varphi \nonumber
\end{align}
on $\Sigma_{t}$, for each $t \in [0,\epsilon)$. Let $Y = W_{t} - \varphi^{-1}\nabla^{\Sigma_{t}}\varphi$. If we multiply the equation (\ref{eq 5.1}) by $\varphi^{-1}$, then we have 
\begin{align} \label{eq 5.2}
\frac{1}{\varphi}\cdot\frac{d\Theta^{+}(t)}{dt} &= {\rm div}_{\Sigma_{t}}Y - \vert Y \vert^{2} + Q + \Theta^{+}(t){\rm tr}_{g}h^{M} - \frac{1}{2}\left(\Theta^{+}(t)\right)^{2} \\
&\leq {\rm div}_{\Sigma_{t}}Y + Q + \Theta^{+}(t){\rm tr}_{g}h^{M} \nonumber
\end{align}
for each $t \in [0,\epsilon)$. Integrating the inequality (\ref{eq 5.2}) over $\Sigma_{t}$, we obtain
\begin{align*}
\left(\Theta^{+}(t)\right)^{'}\int_{\Sigma_{t}}\frac{1}{\varphi} \, dv - \Theta^{+}(t)\int_{\Sigma_{t}}{\rm tr}_{g}h^{M}dv &\leq \int_{\Sigma_{t}}Q \, dv + \int_{\partial\Sigma_{t}}\langle Y, \nu_{t} \rangle \, ds \\
&= \int_{\Sigma_{t}}Q \, dv + \int_{\partial\Sigma_{t}}\left(\langle W_{t},\nu_{t} \rangle - \frac{1}{\varphi}\cdot\frac{\partial\varphi}{\partial\nu_{t}}\right) ds,
\end{align*}
since $\left(\Theta^{+}(t)\right)^{'}$ is also a constant on $\Sigma_{t}$. By the eigenvalue problem (\ref{stable ope}) and the equation (\ref{simple q}), we get
\begin{align*}
\left(\Theta^{+}(t)\right)^{'}\int_{\Sigma_{t}}\frac{1}{\varphi} \, dv  &\leq \Theta^{+}(t)\int_{\Sigma_{t}}{\rm tr}_{g}h^{M}dv + \int_{\Sigma_{t}}Q \, dv - \int_{\partial\Sigma_{t}}\left(q - \langle W_{t},\nu_{t} \rangle\right) ds \\
&\quad + \int_{\partial\Sigma_{t}}H(t)\cot\theta \, ds + \int_{\partial\Sigma_{t}}k^{\partial\Sigma_{t}}ds. \\
&= \Theta^{+}(t)\int_{\Sigma_{t}}{\rm tr}_{g}h^{M}dv + \int_{\Sigma_{t}}Q \, dv + \cot\theta\int_{\partial\Sigma_{t}}\Theta^{+}(t) \, ds  \\
&\quad + \int_{\partial\Sigma_{t}}k^{\partial\Sigma_{t}}ds - \int_{\partial\Sigma_{t}}\frac{1}{\sin\theta}\left(H^{\partial M} + (\cos\theta){\rm tr}_{\Sigma_{t}}h^{M} - \sin\theta\langle W_{t},\nu_{t} \rangle\right) ds.
\end{align*}
By the equation (\ref{tilted}) and the definition of $Q$, we obtain
\begin{align*}
(\Theta^{+}(t))'\int_{\Sigma_{t}}\frac{1}{\varphi}dv &\leq \Theta^{+}(t)\int_{\Sigma_{t}}{\rm tr}_{g}h^{M}dv + \int_{\Sigma_{t}}K^{\Sigma_{t}}dv - \int_{\Sigma_{t}}\left(\mu + J(N_{t})\right)dv \\
&\quad - \frac{1}{2}\int_{\Sigma_{t}}\vert \chi^{+}_{t} \vert^{2}dv + \int_{\partial\Sigma_{t}}k^{\partial\Sigma_{t}}ds + (\cot\theta)\Theta^{+}(t)A(\partial\Sigma_{t}) \\
&\quad -\frac{1}{\sin\theta}\int_{\partial\Sigma_{t}}\left(H^{\partial M} + (\cos\theta){\rm tr}_{\partial M}h^{M} + (\sin\theta) h^{M}(\bar{N},\bar{\nu}_{t})\right)ds \\
&\leq \Theta^{+}(t)\int_{\Sigma_{t}}{\rm tr}_{g}h^{M}dv + \int_{\Sigma_{t}}K^{\Sigma_{t}}dv - \int_{\Sigma_{t}}\left(\mu + J(N_{t})\right)dv \\
&\quad + \int_{\partial\Sigma_{t}}k^{\partial\Sigma_{t}}ds + (\cot\theta)\Theta^{+}(t)A(\partial\Sigma_{t}),
\end{align*}
where we used the tilted dominant boundary energy condition in the last inequality. By the Cauchy-Schwarz inequality, we get
\begin{align} \label{CSI}
\mu + J(N_{t}) \geq \mu  - \vert J \vert \geq C
\end{align}
for $C \in \mathbb{R}$ on $M_{+}$. Using the inequality (\ref{CSI}) and the Gauss-Bonnet theorem, we obtain
\begin{align*}
(\Theta^{+}(t))'\int_{\Sigma_{t}}\frac{1}{\varphi}dv &\leq \Theta^{+}(t)\int_{\Sigma_{t}}{\rm tr}_{g}h^{M}dv + 2\pi\chi(\Sigma) - A(\Sigma_{t})C + \Theta^{+}(t)\cot\theta A(\partial\Sigma_{t}) \\
&= \Theta^{+}(t)\int_{\Sigma_{t}}{\rm tr}_{g}h^{M}dv + C\left(A(\Sigma) - A(\Sigma_{t})\right) + \Theta^{+}(t)\cot\theta A(\partial\Sigma_{t}),
\end{align*}
where we used the equality (\ref{area est 1-1}) of Theorem \ref{main thm 1} in the last equation. By the first variation of area, we have
\begin{align} \label{first var}
A(\Sigma) - A(\Sigma_{t}) &= -\int^{t}_{0}\frac{d}{dr}A(\Sigma_{r})dr \\
&=-\left(\int^{t}_{0}\left(\int_{\Sigma_{r}}H(r)\varphi \, dv\right)dr - \cot\theta\int^{t}_{0}\left(\int_{\partial\Sigma_{r}}\varphi \, ds\right)dr\right) \nonumber \\
&= - \int^{t}_{0}\left(\int_{\Sigma_{r}}\Theta^{+}(t)\varphi \, dv\right)dr + \int^{t}_{0}\left(\int_{\Sigma_{r}}{\rm tr}_{\Sigma_{r}}h^{M}\varphi \, dv\right)dr \nonumber \\
&\quad + \cot\theta\int^{t}_{0}\left(\int_{\partial\Sigma_{r}}\varphi \, ds\right)dr. \nonumber
\end{align}
Then we get
\begin{align} \label{eq 5.4}
(\Theta^{+}(t))'\int_{\Sigma_{t}}\frac{1}{\varphi}dv &\leq \Theta^{+}(t)\int_{\Sigma_{t}}{\rm tr}_{g}h^{M}dv - C\Theta^{+}(t)\int^{t}_{0}\left(\int_{\Sigma_{r}}\varphi \, dv\right)dr \\
&\quad + C\int^{t}_{0}\left(\int_{\Sigma_{r}}{\rm tr}_{\Sigma_{r}}h^{M}dv\right)dr \nonumber \\
&\quad + \cot\theta\left(\Theta^{+}(t)A(\partial\Sigma_{t}) + C\int^{t}_{0}\left(\int_{\partial\Sigma_{r}}\varphi \, ds\right)dr\right). \nonumber
\end{align}
Let 
\begin{itemize}
\item[$\bullet$] $\beta_{1}(t) = \int_{\Sigma_{t}}\varphi^{-1} \, dv$,  
\item[$\bullet$] $\beta_{2}(t) = \int_{\Sigma_{r}}\varphi \, dv$,
\item[$\bullet$] $\beta_{3}(t) = \int_{\partial\Sigma_{r}}\varphi \, ds$.
\end{itemize}

\begin{itemize}
\item[CASE (I).] $C > 0$ and ${\rm tr}_{\pi}h^{M} \leq 0$ with $\theta \in [\frac{\pi}{2},\pi)$.
\end{itemize}
If $C > 0$ and ${\rm tr}_{\pi}h^{M} \geq 0$, then we obtain
\begin{align} \label{ineq 5.5.1}
C\int^{t}_{0}\left(\int_{\Sigma_{r}}{\rm tr}_{\Sigma_{r}}h^{M} dv\right)dr \leq 0.
\end{align}
So we have
\begin{align} \label{ineq 5.5}
(\Theta^{+}(t))'\beta_{1}(t) &\leq \Theta^{+}(t)\int_{\Sigma_{t}}{\rm tr}_{g}h^{M}dv - C\Theta^{+}(t)\int^{t}_{0}\beta_{2}(r)dr \\
&\quad  + \cot\theta\left(\Theta^{+}(t)A(\partial\Sigma_{t}) + C\int^{t}_{0}\beta_{3}(r)dr\right). \nonumber
\end{align}
If $t= 0$, then we have
\begin{align*}
(\Theta^{+}(0))'\beta_{1}(0) \leq 0.
\end{align*}
If $(\Theta^{+}(0)) < 0'$, then $\Theta^{+}(t) < 0$ for all small positive $t$. This is a contradiction, since $\Sigma$ is weakly outermost. So $(\Theta^{+}(0))' = 0$. Let $\beta(t)$ denote the right hand side of the inequality (\ref{ineq 5.5}). A simple calculation yields $\beta'(0) = C\cot\theta\frac{\beta_{3}(0)}{\beta_{1}(0)}$. This implies that
\begin{align*}
(\Theta^{+}(0))'' < \beta'(0) < 0
\end{align*}
if and only if $\theta \neq \frac{\pi}{2}$. This is a contradiction, since $\Sigma$ is weakly outermost. Hence, $\theta$ must be $\frac{\pi}{2}$. It follows that
\begin{align*}
(e^{-\int^{t}_{0}\beta_{4}(r)dr}\Theta^{+}(t))' \leq 0,
\end{align*}
where $\beta_{4}(r) = \frac{\left(\int_{\Sigma_{r}}{\rm tr}_{g}h^{M}dv - C\int^{t}_{0}\beta_{2}(r)dr\right)}{\beta_{1}(r)}$. This gives that $e^{-\int^{t}_{0}\beta_{4}(r)dr}\Theta^{+}(t) \leq \Theta^{+}(0)$ for $t \in [0,\epsilon)$. Hence, we conclude that $\Theta^{+}(t) = 0$ for every $t \in [0,\epsilon)$, since $\Sigma$ is weakly outermost. Therefore, all of the above inequalities become equalities. Hence, by the inequality (\ref{ineq 5.5.1}), we get ${\rm tr}_{\Sigma_{t}}h^{M} = 0$ and $H(t) = 0$ on $\Sigma_{t}$. Furthermore, $\Theta^{-}(t) = 0$ on $\Sigma_{t}$ for $t \in [0,\epsilon)$.

\begin{itemize}
\item[CASE (II).] $C < 0$ and ${\rm tr}_{\pi}h^{M} \geq 0$ with $\theta \in (0,\frac{\pi}{2}]$.
\end{itemize}
By the inequality (\ref{eq 5.4}), we obtain
\begin{align} \label{ineq 5.7.1}
C\int^{t}_{0}\left(\int_{\Sigma_{r}}{\rm tr}_{\Sigma_{r}}h^{M}dv\right)dr \leq 0.
\end{align}
Then we get the inequality (\ref{ineq 5.5}), which gives that $(\Theta^{+}(0))' = 0$. Hence, $(\Theta^{+}(0))'' \leq \beta'(0) < 0$, it follows that $\theta$ must be $\frac{\pi}{2}$, since $\Sigma$ is weakly outermost. With the same argument as CASE (1), we deduce that $\Theta^{+}(t) = 0$ for $t \in [0,\epsilon)$. By the inequality (\ref{ineq 5.7.1}), we have ${\rm tr}_{\Sigma_{t}}h^{M} = 0$ and $H(t) = 0$ on $\Sigma_{t}$. Hence, $\Theta^{-}(t) = 0$ for $t \in [0,\epsilon)$.

\begin{itemize}
\item[CASE (III).] $C = 0$, ${\rm tr}_{\pi}h^{M} \geq 0$, and $\Sigma$ is an energy-minimizing.
\end{itemize}
By the inequality (\ref{eq 5.4}), we have
\begin{align*}
(\Theta^{+}(t))'\beta_{1}(t) \leq \Theta^{+}(t)\int_{\Sigma_{r}}{\rm tr}_{g}h^{M}dv + \Theta^{+}(t)(\cot\theta)A(\partial\Sigma_{t}).
\end{align*}
If $t = 0$, then $(\Theta^{+}(0))' \leq 0$. This gives that $(\Theta^{+}(0))' = 0$, since $\Sigma$ is weakly outermost. Let $\beta_{5}(r) = \frac{\left(\int_{\Sigma_{r}}{\rm tr}_{g}h^{M}dv + (\cot\theta)A(\partial\Sigma_{r})\right)}{\beta_{1}(r)}$. Then we have
\begin{align*}
\left(e^{-\int^{t}_{0}\beta_{5}(r)dr}\Theta^{+}(t)\right)' \leq 0.
\end{align*}
It follows that $\Theta^{+}(t) = 0$ for $t \in [0,\epsilon)$. By the energy-minimizing, we obtain
\begin{align} \label{energy-mini}
0 \leq E(t) - E(0) = \int^{t}_{0}E'(r)dr = \int^{t}_{0}\left(\int_{\Sigma_{r}}H(r)\varphi \, dv\right)dr.
\end{align}
Then we have $H(t) \geq 0$ for $t \in [0,\epsilon)$. This implies that $0 \leq H(t) \leq H(t) + {\rm tr}_{\Sigma_{t}}h^{M} = \Theta^{+}(t) = 0$, since $\Sigma$ is 2-convex. So we conclude that $H(t) = 0$ and ${\rm tr}_{\Sigma_{t}}h^{M} = 0$ on $\Sigma_{t}$, which gives that $\Theta^{-}(t) = 0$.

The remaining part of the proof for the case $\theta = \frac{\pi}{2}$ can be obtained by following the argument in Theorem B in \cite{AM} or Section 5 in \cite{M2}.

Finally, we prove that $M$ is flat if $C = 0$, $H^{\partial M} \geq 0$, and $\theta \neq \frac{\pi}{2}$. From the inequality (\ref{eq 5.2}) and $Q = 0$, we have ${\rm div}_{\Sigma_{t}}Y - \vert Y \vert^{2} = 0$. So we get
 \begin{align*}
\int_{\Sigma_{t}}\vert Y \vert^{2}dv &= \int_{\Sigma_{t}}{\rm div}_{\Sigma_{t}}Y \, dv \\
&= - \int_{\partial\Sigma_{t}}\left(q - \langle W_{t},\nu_{t} \rangle\right)ds = 0,
 \end{align*}
where we used item (3) of Theorem \ref{main thm 1-1}. So we obtain $W_{t} = \varphi^{-1}\nabla^{\Sigma_{t}}\varphi$. Now we replace $\Theta^{+}(t)$ and $\varphi$ by $\Theta^{-}(t)$ and $\varphi^{-} = -\varphi$, respectively. Then we have
\begin{align*}
0 = \frac{d\Theta^{-}(t)}{dt} = -\Delta_{\Sigma_{t}}\varphi^{-} + 2\langle W^{-}_{t}, \nabla^{\Sigma_{t}}\varphi^{-} \rangle + \left(Q^{-} - \vert W^{-}_{t} \vert^{2}  + {\rm div}_{\Sigma_{t}}W^{-}_{t} \right)\varphi^{-}.
\end{align*}
Note that
\begin{align*}
Q^{-} &= K^{\Sigma_{t}} - \left(\mu + J(-N_{t}) - \frac{1}{2}\vert \chi^{-}_{t} \vert^{2}\right) \\
&= -\left(\mu + \vert J \vert - \frac{1}{2}\vert \chi^{-}_{t} \vert^{2}\right)     \\
&= -2\vert J \vert - \frac{1}{2}\vert \chi^{-}_{t} \vert^{2}
\end{align*}
and
\begin{align*}
W^{-}_{t} = -W_{t} = -\varphi^{-1}\nabla^{\Sigma_{t}}\varphi.
\end{align*}
Then we obtain
\begin{align} \label{eq 5.8}
0 = \Delta_{\Sigma_{t}}\varphi + \varphi^{-1}\vert \nabla^{\Sigma_{t}}\varphi \vert^{2} + \varphi\left(\vert J \vert + \frac{1}{4}\vert \chi^{-}_{t} \vert^{2}\right).
\end{align}
By the eigenvalue problem (\ref{stable ope}) and the equation (\ref{simple q}), we get
\begin{align*}
\frac{\partial\varphi}{\partial\nu_{t}} = q\varphi = \frac{\varphi}{\sin\theta}H^{\partial M} \geq 0.
\end{align*}
Integrating the equation (\ref{eq 5.8}) over $\Sigma_{t}$, we obtain
\begin{align*}
0 \leq \int_{\partial\Sigma_{t}}\frac{\partial\varphi}{\partial\nu_{t}}ds = \int_{\Sigma_{t}} \Delta_{\Sigma_{t}}\varphi \, dv = -\int_{\Sigma_{t}}\left(\frac{\vert \nabla^{\Sigma_{t}}\varphi\vert^{2}}{\varphi}  + \left(\vert J \vert + \frac{1}{4}\vert \chi^{-}_{t} \vert^{2} \right)\varphi \right)dv \leq 0.
\end{align*}
This shows that $\vert \nabla^{\Sigma_{t}}\varphi \vert = \vert J \vert = \vert \chi^{-}_{t} \vert = 0$. Thus, $\varphi$ is a constant on $\Sigma_{t}$ and $H^{\partial M} = 0$ along $\partial\Sigma_{t}$. Since $W_{t} = \varphi^{-1}\nabla^{\Sigma_{t}}\varphi$, we have $h^{M}(N_{t},\cdot)\vert_{\Sigma_{t}} = 0$. Furthermore, we get $h^{M}\vert_{\Sigma_{t}} = 0$, since $\chi^{+}_{t} = \chi^{-}_{t} = 0$. Hence, we conclude that $h^{\Sigma_{t}} = 0$ for $t \in [0,\epsilon)$. Since $\Sigma_{t}$ is flat and totally geodesic, the Gauss equation shows that the sectional curvature of $M$ is zero. Hence, $M$ is flat with $H^{\partial M} = 0$ around of $\Sigma$.

\end{proof}

\section{High dimensional initial data set}
In this section, we extend the results of Section 4 to high dimensional initial data sets. 

To generalize our results to high dimensions, we first introduce a geometric or topological invariant that serves as a replacement for the Euler characteristic, which is not applicable in high dimensions. Let $(\Sigma,\gamma)$ be a compact $n$-dimensional Riemannian manifold ($n \geq 3$) with boundary $\partial\Sigma$. For $(\alpha_{1},\alpha_{2}) \in \mathbb{R} \times \mathbb{R} - \{(0,0)\}$ and for any positive function $\rho \in C^{\infty}(\Sigma)$, we define the \textit{Yamabe functional} as
\begin{align} \label{Yamabe func}
\mathcal{Q}^{\alpha_{1},\alpha_{2}}(\rho) = \frac{\int_{\Sigma}\left(\frac{4(n-1)}{n-2}\vert \nabla^{\Sigma}\rho \vert^{2} + R^{\Sigma}\rho^{2}\right)dv + 2\int_{\partial\Sigma}H^{\partial\Sigma}\rho^{2}ds}{\left(\alpha_{1}\left(\int_{\Sigma}\rho^{\frac{2n}{n-2}}dv\right) + \alpha_{2}\left(\int_{\partial\Sigma}\rho^{\frac{2(n-1)}{n-2}}ds\right)^{\frac{n}{n-1}}\right)^{\frac{n-2}{n}}},
\end{align}
where $R^{\Sigma}$ is the scalar curvature of $\Sigma$ and $H^{\partial\Sigma}$ is the mean curvature of $\partial\Sigma$. The corresponding \textit{Yamabe invariant} is defined by
\begin{align} \label{Yamabe inva}
\mathcal{Y}^{\alpha_{1},\alpha_{2}}(\Sigma,\partial\Sigma) := \inf_{\rho \in C^{\infty}(\Sigma), \,\rho > 0} \mathcal{Q}^{\alpha_{1},\alpha_{2}}(\rho)
\end{align}
for $(\alpha_{1},\alpha_{2}) \in \{(1,0), (0,1)\}$. This quantity is invariant under conformal transformations (see \cite{ES1, ES2, ES3, ES4}). It is well-known that $\mathcal{Y}^{1,0}(\Sigma,\partial\Sigma) \leq \mathcal{Y}(\mathbb{S}^{n}_{+},\partial\mathbb{S}^{n}_{+}) = n(n-1)\left(\frac{V(\mathbb{S}^{n})}{2}\right)^{\frac{2}{n}}$ and $\mathcal{Y}^{0,1}(\Sigma,\partial\Sigma) \leq \mathcal{Y}(\mathbb{B}^{n},\partial\mathbb{B}^{n}) = 2(n-1)V(\mathbb{S}^{n-1})^{\frac{1}{n-1}}$, where $\mathbb{S}^{n}_{+}$, $\mathbb{S}^{n}$, and $\mathbb{B}^{n}$ are the upper standard hemisphere, standard sphere, and unit ball in $\mathbb{R}^{n}$, respectively (cf. \cite{CHE, ES2, ES3, WA}). In \cite{ES2, ES3}, Escobar showed the following: If $\mathcal{Y}^{1,0}(\Sigma,\partial\Sigma) < \mathcal{Y}(\mathbb{S}^{n}_{+},\partial\mathbb{S}^{n}_{+})$, then there exists a conformal metric on $\Sigma$ with constant scalar curvature and zero mean curvature along the boundary. Similarly, if $\mathcal{Y}^{0,1}(\Sigma,\partial\Sigma) < \mathcal{Y}(\mathbb{B}^{n},\partial \mathbb{B}^{n})$, then there exists a conformal metric with zero scalar curvature and constant mean curvature on the boundary. We now define the \textit{Yamabe constant}, which serves as a generalization of the Euler characteristic in high dimensions.
\begin{definition}
Let $[\gamma]$ and $\mathcal{C}(\Sigma)$ denote the conformal class of the metric $\gamma$ and the space of all conformal classes on $\Sigma$, respectively. The \textit{Yamabe constant} of a compact Riemannian manifold $\Sigma$ with boundary $\partial\Sigma$ is defined by
\begin{align} \label{Yamabe cons}
\sigma^{\alpha_{1},\alpha_{2}}(\Sigma,\partial\Sigma) = \sup_{[\gamma] \in \mathcal{C}(\Sigma)} \mathcal{Y}^{\alpha_{1},\alpha_{2}}(\Sigma,\partial\Sigma).
\end{align}
\end{definition}
In the case of Riemannian surface with boundary, the scalar curvature is twice the Gaussian curvature, and the mean curvature of the boundary coincides with the geodesic curvature. Consequently, it follows that $\sigma^{1,0}(\Sigma,\partial\Sigma) = 4\pi\chi(\Sigma)$ (cf. \cite{BC, CG, MO, MO2}). Similarly, $\sigma^{0,1}(\Sigma,\partial\Sigma) = 4\pi\chi(\Sigma)$. Thus, the Yamabe constant may be regarded as a natural generalization of the Euler characteristic to high dimensions.

We now establish an area estimate of an $n$-dimensional MOTS with capillary boundary, using the Yamabe constant.

\begin{theorem} [Theorem \ref{main high 1}] \label{main high 1-1}
Let $(M,g,h^{M})$ be an $(n+1)$-dimensional initial data set with boundary and $\Sigma^{n}$ be a compact stable MOTS with capillary boundary in $(M,g,h^{M})$. Assume that $\mu + J(N) \geq D$ for $D \in \mathbb{R}^{-} \cup \{0\}$ and $(M,g,h^{M})$ satisfies the tilted dominant boundary energy condition.
\begin{itemize}
\item[(A)] If $D < 0$ and $\sigma^{1,0}(\Sigma,\partial\Sigma) < 0$, then
\begin{align} \label{high area est 1}
2DA(\Sigma)^{\frac{2}{n}} \leq \sigma^{1,0}(\Sigma,\partial\Sigma).
\end{align}
\item[(B)] If $D = 0$ and $\sigma^{1,0}(\Sigma,\partial\Sigma) \leq 0$, then $\sigma^{1,0}(\Sigma,\partial\Sigma) = 0$.
\end{itemize}
Moreover, if equality holds in (\ref{high area est 1}) or $\sigma^{1,0}(\Sigma,\partial\Sigma) = 0$, then we have
\begin{itemize}
\item[(1)] $\mu + J(N) = D$ and $\chi^{+} = 0$ on $\Sigma$,
\item[(2)] $\Sigma$ is Einstein manifold with $R^{\Sigma} = 2D$ and totally geodesic boundary,
\item[(3)] $\lambda_{1}(L_{s}) = \lambda_{1}(L) = 0$ on $\Sigma$,
\item[(4)] $H^{\partial M} + (\cos\theta){\rm tr}_{\partial M}h^{M} = \sin\theta\vert h^{M}(\bar{N},\cdot)^{T} \vert$ along $\partial\Sigma$.
\end{itemize}
\end{theorem}

\begin{proof}
Since the equation (\ref{sym eigen}), we have
\begin{align*}
0 &\leq 2\int_{\Sigma}\left(\vert \nabla^{\Sigma}\varphi \vert^{2} + Q\varphi^{2}\right)dv -2\int_{\partial\Sigma}\left(q - \langle W,\nu \rangle\right)\varphi^{2}ds,
\end{align*}
where $\varphi \in C^{\infty}(\Sigma)$. By the the definition of $Q$ and the equation (\ref{simple q}), we obtain
\begin{align} \label{ine 5.2}
0 &\leq \int_{\Sigma}\left(2\vert \nabla^{\Sigma}\varphi \vert^{2} +\left(R^{\Sigma} -2\left(\mu + J(N)\right) - \vert \chi^{+} \vert^{2}\right)\varphi^{2}\right)dv \\
&\quad -2\int_{\partial\Sigma}\frac{1}{\sin\theta}\left(H^{\partial M} - (\cos\theta)H - (\sin\theta)H^{\partial\Sigma} - \sin\theta\langle W,\nu \rangle\right)\varphi^{2}ds \nonumber \\
&\leq \int_{\Sigma}\left(2\vert \nabla^{\Sigma}\varphi \vert^{2} + \left(R^{\Sigma} - 2D\right)\varphi^{2}\right)dv \nonumber \\
&\quad -2\int_{\partial\Sigma}\frac{1}{\sin\theta}\left(H^{\partial M} + (\cos\theta){\rm tr}_{\Sigma}h^{M} - (\sin\theta)H^{\partial\Sigma} - (\sin\theta)h^{M}(N,\nu)\right)\varphi^{2}ds, \nonumber
\end{align}
where we used $\mu + J(N) \geq D$. If we choose an orthonormal frame $\{e_{1}, \cdots,e_{n-1},e_{n} = \nu\} \in T\Sigma$, then by an argument analogous to that used in deriving equation \eqref{tilted}, we obtain
\begin{align} \label{high tilted}
(\cos\theta){\rm tr}_{\Sigma}h^{M} - (\sin\theta)h^{M}(N,\nu) = (\cos\theta){\rm tr}_{\partial M}h^{M} + (\sin\theta)h^{M}(\bar{N},\bar{\nu}).
\end{align}
Substituting (\ref{high tilted}) into (\ref{ine 5.2}) and using the tilted dominant boundary energy condition, we obtain
\begin{align} \label{ine 5.4}
0 &\leq \int_{\Sigma}\left(2\vert \nabla^{\Sigma}\varphi \vert^{2} + \left(R^{\Sigma} -2D\right)\varphi^{2}\right)dv + 2\int_{\partial\Sigma}H^{\partial\Sigma}\varphi^{2}ds \\
&\leq \int_{\Sigma}\left(\frac{4(n-1)}{n-2}\vert \nabla^{\Sigma}\varphi \vert^{2} + R^{\Sigma}\varphi^{2}\right)dv + 2\int_{\partial\Sigma}H^{\partial\Sigma}\varphi^{2}ds -2D\int_{\Sigma}\varphi^{2}dv, \nonumber
\end{align}
where we used $2 < \frac{4(n-1)}{n-2}$ for any $n \geq 3$ in the last inequality.
\begin{itemize}
\item[CASE (A).] $D < 0$ and $\sigma^{1,0}(\Sigma,\partial\Sigma) < 0$.
\end{itemize}
By the H\"older's inequality, we obtain
\begin{align} \label{ine 5.5}
-2D \int_{\Sigma}\varphi^{2}dv \leq -2D\left(\int_{\Sigma}\varphi^{\frac{2n}{n-2}}dv\right)^{\frac{n-2}{n}}A(\Sigma)^{\frac{2}{n}}.
\end{align}
Putting (\ref{ine 5.5}) into (\ref{ine 5.4}), we get
\begin{align} \label{ine 5.6}
0 &\leq \int_{\Sigma}\left(\frac{4(n-1)}{n-2}\vert \nabla^{\Sigma}\varphi \vert^{2} + R^{\Sigma}\varphi^{2}\right)dv + 2\int_{\partial\Sigma}H^{\partial\Sigma}\varphi^{2}ds \\
&\quad -2D\left(\int_{\Sigma}\varphi^{\frac{2n}{n-2}}dv\right)^{\frac{n-2}{n}}A(\Sigma)^{\frac{2}{n}}. \nonumber
\end{align}
Dividing the inequality (\ref{ine 5.6}) by $\left(\int_{\Sigma}\varphi^{\frac{2n}{n-2}}dv\right)^{\frac{n-2}{n}}$, we have
\begin{align*}
2DA(\Sigma)^{\frac{2}{n}} &\leq \frac{\int_{\Sigma}\left(\frac{4(n-1)}{n-2}\vert \nabla^{\Sigma}\varphi \vert^{2} + R^{\Sigma}\varphi^{2}\right)dv + 2\int_{\partial\Sigma}H^{\partial\Sigma}\varphi^{2}ds}{\left(\int_{\Sigma}\varphi^{\frac{2n}{n-2}}dv\right)^{\frac{n-2}{n}}} \\
&\leq \mathcal{Y}^{1,0}(\Sigma,\partial\Sigma) \leq \sigma^{1,0}(\Sigma,\partial\Sigma).
\end{align*}
\begin{itemize}
\item[CASE (B).] $D = 0$ and $\sigma^{1,0}(\Sigma,\partial\Sigma) \leq 0$.
\end{itemize}
Dividing the inequality (\ref{ine 5.4}) by $\left(\int_{\Sigma}\varphi^{\frac{2n}{n-2}}dv\right)^{\frac{n-2}{n}}$, we obtain
\begin{align*}
0 &\leq \frac{\int_{\Sigma}\left(\frac{4(n-1)}{n-2}\vert \nabla^{\Sigma}\varphi \vert^{2} + R^{\Sigma}\varphi^{2}\right)dv + 2\int_{\partial\Sigma}H^{\partial\Sigma}\varphi^{2}ds}{\left(\int_{\Sigma}\varphi^{\frac{2n}{n-2}}dv\right)^{\frac{n-2}{n}}} \\
&\leq \sigma^{1,0}(\Sigma,\partial\Sigma) \leq 0.
\end{align*}
Hence, we conclude that $\sigma^{1,0}(\Sigma,\partial\Sigma) = 0$.

If equality holds in (\ref{high area est 1}) or $\sigma^{1,0}(\Sigma,\partial\Sigma) = 0$, then all of the above inequalities become equalities. From the equality (\ref{ine 5.2}) and (\ref{ine 5.4}), it follows that $\mu + J(N) = D$, $R^{\Sigma} = 2D$, and $\chi^{+} = 0$ on $\Sigma$, and $H^{\partial M} + (\cos\theta){\rm tr}_{\partial M}h^{M} = \sin\theta\vert h^{M}(\bar{N},\cdot)^{T}\vert = -(\sin\theta)h^{M}(\bar{N},\bar{\nu})$ along $\partial\Sigma$. Since $2 < \frac{4(n-1)}{n-2}$ for any $n \geq 3$, it follows from the equality (\ref{ine 5.4}) that $\varphi$ is a constant. Moreover, $Q = \frac{R^{\Sigma}}{2} - (\mu + J(N)) - \frac{\vert \chi^{+} \vert^{2}}{2} = 0$ on $\Sigma$ and $\lambda_{1}(L_{s}) = \lambda_{1}(L) = 0$. From the Theorem 2.1 in \cite{AR} or the Proposition 5.3 in \cite{CS}, $\Sigma$ is Einstein manifold with totally geodesic boundary. So we have $q - \langle W,\nu \rangle = 0$ along $\partial\Sigma$.

\end{proof}

If we assume that $2DA(\Sigma)^{\frac{2}{n}} = \sigma^{1,0}(\Sigma,\partial\Sigma)$ or $\sigma^{1,0}(\Sigma,\partial\Sigma) = 0$, then conclusions {\rm (1), (2), (3),} and {\rm  (4)} of Theorem \ref{main high 1-1} follow. Consequently, we obtain a high dimensional analogue of Lemma \ref{lem 4.1}. Therefore, an analogue of Proposition \ref{foliation} also holds in the high dimensional setting.

Next, we prove a rigidity result for an $(n+1)$-dimensional initial dats set.

\begin{theorem} [Theorem \ref{main high 2}] \label{main high 1-2}
Let $(M,g,h^{M})$ be an $(n+1)$-dimensional initial data sets with boundary and $\Sigma^{n}$ be a compact stable MOTS with capillary boundary in $(M,g,h^{M})$ that is weakly outermost in $M_{+}$. Suppose that $\mu - \vert J \vert \geq D$ for $D \in \mathbb{R}^{-} \cup \{0\}$ and $(M,g,h^{M})$ satisfies the tilted dominant boundary energy condition, all of which hold on $M_{+}$.
\begin{itemize}
\item[(A)] If $D < 0$, $\sigma^{1,0}(\Sigma,\partial\Sigma) < 0$, then
\end{itemize}
\begin{align*}
2DA(\Sigma)^{\frac{2}{n}} \leq \sigma^{1,0}(\Sigma,\partial\Sigma).
\end{align*}
Moreover, if equality holds and $n$-convex with $\theta \in (0,\frac{\pi}{2}]$ , then $\theta$ must be $\frac{\pi}{2}$.
\begin{itemize}
\item[(B)] If $D = 0$, $\sigma^{1,0}(\Sigma,\partial\Sigma) \leq 0$, $n$-convex, and an energy-minimizing, then $\theta = \frac{\pi}{2}$ or not.
\end{itemize}
Furthermore, the case of $\theta = \frac{\pi}{2}$, the following hold:
\begin{itemize}
 \item[(i)] there exists an outer neighborhood $V$ of $\Sigma$ in $M$ such that $(V,g\vert_{V})$ is isometric to $([0,\epsilon) \times \Sigma, dt^{2} + \gamma)$, where $\gamma = g|_{\Sigma}$ and $(\Sigma,\gamma)$ is the Einstein manifold with the scalar curvature $R^{\Sigma} = 2D$ and a totally geodesic boundary;
\item[(ii)] $h^{M} = a dt^{2}$ on $V$, where $a \in C^{\infty}$ depends only on $t \in [0,\epsilon)$;
\item[(iii)] $\mu = D$ and $J = 0$ on $V$;
\item[(iv)] The tilted dominant boundary condition is saturated along $V \cap \partial M$.
\end{itemize}
Otherwise, if without loss of generality of {\rm (B)}, $Ric^{M} = \frac{\mu}{n+1}g$, $H^{\partial M} \geq 0$, and $\theta \neq \frac{\pi}{2}$, then $\Sigma$ is Ricci flat with totally geodesic boundary and $M$ is Ricci flat with $H^{\partial M} = 0$ in an outer neighborhood of $\Sigma$.
\end{theorem}

\begin{proof}
From (\ref{eq 5.2}), we have
\begin{align} \label{ine 5.7}
\frac{1}{\varphi}\cdot\frac{d\Theta^{+}(t)}{dt} \leq {\rm div}_{\Sigma_{t}}Y - \vert Y \vert^{2} + Q + \Theta^{+}(t){\rm tr}_{g}h^{M}
\end{align}
on $\Sigma_{t}$, for $t \in [0,\epsilon)$, where $Y = W_{t} - \varphi^{-1}\nabla^{\Sigma_{t}}\varphi$. Let $u_{t} \in C^{\infty}(\Sigma_{t})$. Multiplying the inequality (\ref{ine 5.7}) by $u_{t}^{2}$, we obtain
\begin{align} \label{ine 5.8}
\frac{u_{t}^{2}}{\varphi}\frac{d\Theta^{+}(t)}{dt} &\leq u_{t}^{2}{\rm div}_{\Sigma_{t}}Y - u^{2}_{t}\vert Y \vert^{2} + Qu^{2}_{t} + u^{2}_{t}\Theta^{+}(t){\rm tr}_{g}h^{M} \\
&= {\rm div}_{\Sigma_{t}}(u^{2}_{t}Y) - 2u_{t}\langle \nabla^{\Sigma_{t}}u_{t}, Y \rangle +Qu^{2}_{t} + u^{2}_{t}\Theta^{+}(t){\rm tr}_{g}h^{M} \nonumber \\
&\leq {\rm div}_{\Sigma_{t}}(u^{2}_{t}Y) + \vert \nabla^{\Sigma_{t}}u_{t} \vert^{2} + Qu^{2}_{t} + u^{2}_{t}\Theta^{+}(t){\rm tr}_{g}h^{M} \nonumber \\
&\leq {\rm div}_{\Sigma_{t}}(u^{2}_{t}Y) + \vert \nabla^{\Sigma_{t}}u_{t} \vert^{2} + u^{2}_{t}\left(\frac{R^{\Sigma_{t}}}{2} - D\right) + u^{2}_{t}\Theta^{+}(t){\rm tr}_{g}h^{M}, \nonumber
\end{align}
where we used our dominant energy condition in the last inequality. Integrating (\ref{ine 5.8}) over $\Sigma_{t}$, then we have
\begin{align*}
(\Theta^{+}(t))'\int_{\Sigma_{t}}\frac{u^{2}_{t}}{\varphi}dv &\leq \int_{\Sigma_{t}}\left(\vert \nabla^{\Sigma_{t}}u_{t} \vert^{2} + \frac{R^{\Sigma_{t}}}{2}u^{2}_{t}\right)dv -D\int_{\Sigma_{t}}u^{2}_{t}dv + \Theta^{+}(t)\int_{\Sigma_{t}}{\rm tr}_{g}h^{M}u^{2}_{t}dv \\
&\quad + \int_{\partial\Sigma_{t}}u^{2}_{t}\langle Y,\nu_{t} \rangle ds.
\end{align*}
Note that
\begin{align*}
\int_{\partial\Sigma_{t}}u^{2}_{t}\langle Y, \nu_{t} \rangle ds &= \int_{\partial\Sigma_{t}} u^{2}_{t}\left(\langle W_{t},\nu_{t} \rangle - \frac{1}{\varphi}\langle \nabla^{\Sigma_{t}}\varphi, \nu_{t} \rangle\right) ds \\
&= \int_{\partial\Sigma_{t}}u^{2}_{t} \left(\langle W_{t},\nu_{t} \rangle - q\right) ds \\
&= -\int_{\partial\Sigma_{t}}\frac{u^{2}_{t}}{\sin\theta}\left(H^{\partial M} - (\cos\theta)\Theta^{+}(t) - (\sin\theta)H^{\partial\Sigma_{t}}\right) ds \\
&\quad -\int_{\partial\Sigma_{t}} \frac{u^{2}_{t}}{\sin\theta}\left((\cos\theta){\rm tr}_{\Sigma_{t}}h^{M} - (\sin\theta)h^{M}(N_{t},\nu_{t})\right) ds \\
&= -\int_{\partial\Sigma_{t}}\frac{u^{2}_{t}}{\sin\theta}\left(H^{\partial M}  + (\cos\theta){\rm tr}_{\partial M}h^{M} + (\sin\theta)h^{M}(\bar{N},\bar{\nu}_{t})\right) ds \\
&\quad +\int_{\partial\Sigma_{t}}\frac{u^{2}_{t}}{\sin\theta}\left((\cos\theta)\Theta^{+}(t) + (\sin\theta)H^{\partial\Sigma_{t}}\right) ds \\
&\leq \int_{\partial\Sigma_{t}}\left(u^{2}_{t}(\cot\theta)\Theta^{+}(t) + H^{\partial\Sigma_{t}}u^{2}_{t}\right) ds,
\end{align*}
where we used the equality (\ref{high tilted}) in the fourth equality and tilted dominant boundary energy condition in the last inequality (see also appendix A in \cite{CH}). Then we obtain
\begin{align} \label{ine 5.9}
2(\Theta^{+}(t))'\int_{\Sigma_{t}}\frac{u^{2}_{t}}{\varphi}dv &\leq \int_{\Sigma_{t}}\left(2\vert \nabla^{\Sigma_{t}}u_{t} \vert^{2} + R^{\Sigma_{t}}u^{2}_{t}\right)dv + 2\int_{\partial\Sigma_{t}}H^{\partial\Sigma_{t}}u^{2}_{t}ds \\
&\quad -2D\int_{\Sigma_{t}}u^{2}_{t}dv + 2\Theta^{+}(t)\int_{\Sigma_{t}}{\rm tr}_{g}h^{M}u^{2}_{t}dv \nonumber  \\
&\quad + 2(\cot\theta)\Theta^{+}(t)\int_{\partial\Sigma_{t}}u^{2}_{t}ds. \nonumber \\
&\leq \int_{\Sigma_{t}}\left(\frac{4(n-1)}{n-2}\vert \nabla^{\Sigma_{t}}u_{t} \vert^{2} + R^{\Sigma_{t}}u^{2}_{t}\right)dv + 2\int_{\partial\Sigma_{t}}H^{\partial\Sigma_{t}}u^{2}_{t}ds \nonumber \\
&\quad -2D\int_{\Sigma_{t}}u^{2}_{t}dv + 2\Theta^{+}(t)\int_{\Sigma_{t}}{\rm tr}_{g}h^{M}u^{2}_{t}dv \nonumber  \\
&\quad + 2(\cot\theta)\Theta^{+}(t)\int_{\partial\Sigma_{t}}u^{2}_{t}ds. \nonumber
\end{align}
Let $u_{t}$ be a positive smooth function on $\Sigma_{t}$ satisfying $\tilde{\gamma}_{t} = u^{\frac{4}{n-2}}_{t}\gamma_{t}$ for each $t \in [0,\epsilon)$. From the solution to the Yamabe problem with boundary, there exists the metric $\tilde{\gamma}_{t} = u^{\frac{4}{n-2}}_{t}\gamma_{t}$ has constant scalar curvature $2D$ and vanishing mean curvature of the boundary (\cite{AR, ES2, ES4}).

\begin{itemize}
\item[CASE (A).] $D < 0$, $\sigma^{1,0}(\Sigma,\partial\Sigma) < 0$, and $n$-convex with $\theta \in (0,\frac{\pi}{2}]$.
\end{itemize}
By the inequality (\ref{ine 5.5}) and (\ref{ine 5.9}), we have
\begin{align} \label{ine 5.10}
2(\Theta^{+}(t))'\int_{\Sigma_{t}}\frac{u^{2}_{t}}{\varphi}dv &\leq \int_{\Sigma_{t}}\left(\frac{4(n-1)}{n-2}\vert \nabla^{\Sigma_{t}}u_{t} \vert^{2} + R^{\Sigma_{t}}u^{2}_{t}\right)dv + 2\int_{\partial\Sigma_{t}}H^{\partial\Sigma_{t}}u^{2}_{t}ds  \\
&\quad + 2\Theta^{+}(t)\int_{\Sigma_{t}}{\rm tr}_{g}h^{M}u^{2}_{t}dv -2D\left(\int_{\Sigma_{t}}u^{\frac{2n}{n-2}}_{t}dv\right)^{\frac{n-2}{n}}A(\Sigma_{t})^{\frac{2}{n}} \nonumber\\
&\quad + 2(\cot\theta)\Theta^{+}(t)\int_{\partial\Sigma_{t}}u^{2}_{t}ds. \nonumber
\end{align}
Dividing the inequality (\ref{ine 5.10}) by $\left(\int_{\Sigma_{t}}u^{\frac{2n}{n-2}}_{t}dv\right)^{\frac{n-2}{n}}$, we obtain
\begin{align*}
2(\Theta^{+}(t))'\beta_{6}(t) &\leq \frac{\int_{\Sigma_{t}}\left(\frac{4(n-1)}{n-2}\vert \nabla^{\Sigma_{t}}u_{t} \vert^{2} + R^{\Sigma_{t}}u^{2}_{t}\right)dv + 2\int_{\partial\Sigma_{t}}H^{\partial\Sigma_{t}}u^{2}_{t}ds}{\left(\int_{\Sigma_{t}}u^{\frac{2n}{n-2}}_{t}dv\right)^{\frac{n-2}{n}}} \\
&\quad +2(\cot\theta)\Theta^{+}(t)\beta_{7}(t) + 2\Theta^{+}(t)\beta_{8}(t) -2DA(\Sigma_{t})^{\frac{2}{n}},
\end{align*}
where $\beta_{6}(t) = \frac{\int_{\Sigma_{t}}\frac{u^{2}_{t}}{\varphi}dv}{\left(\int_{\Sigma_{t}}u^{\frac{2n}{n-2}}_{t}dv\right)^{\frac{n-2}{n}}}$, $\beta_{7}(t) = \frac{\int_{\partial\Sigma_{t}}u^{2}_{t}ds}{\left(\int_{\Sigma_{t}}u_{t}^{\frac{2n}{n-2}}dv\right)^{\frac{n-2}{n}}}$, and $\beta_{8}(t) = \frac{\int_{\Sigma_{t}}{\rm tr}_{g}h^{M}u^{2}_{t}dv}{\left(\int_{\Sigma_{t}}u_{t}^{\frac{2n}{n-2}}dv\right)^{\frac{n-2}{n}}}$. By the definition of the Yamabe constant, we have
\begin{align*}
2(\Theta^{+}(t))'\beta_{6}(t) &\leq \sigma^{1,0}(\Sigma,\partial\Sigma) +2(\cot\theta)\Theta^{+}(t)\beta_{7}(t) + 2\Theta^{+}(t)\beta_{8}(t) -2DA(\Sigma_{t})^{\frac{2}{n}}.
\end{align*}
Since the equality (\ref{high area est 1}) in Theorem \ref{main high 1-1}, we get
\begin{align} \label{ine 5.11}
2(\Theta^{+}(t))'\beta_{6}(t) &\leq 2D\left(A(\Sigma)^{\frac{2}{n}} - A(\Sigma_{t})^{\frac{2}{n}}\right) +2(\cot\theta)\Theta^{+}(t)\beta_{7}(t) + 2\Theta^{+}(t)\beta_{8}(t) \\
&= -\frac{4D}{n}\int^{t}_{0}\frac{d}{dr}\left(A(\Sigma_{r})\right)A(\Sigma_{r})^{\frac{2-n}{n}}dr +2(\cot\theta)\Theta^{+}(t)\beta_{7}(t) \nonumber \\
&\quad + 2\Theta^{+}(t)\beta_{8}(t). \nonumber
\end{align}
By the first variation (\ref{first var}) for the area, we obtain
\begin{align*}
-\frac{4D}{n}\int^{t}_{0}\frac{d}{dr}\left(A(\Sigma_{r})\right)A(\Sigma_{r})^{\frac{2-n}{n}}dr &= -\frac{4D}{n}\Theta^{+}(t)\int^{t}_{0}A(\Sigma_{r})^{\frac{2-n}{n}}\left(\int_{\Sigma_{r}}\varphi \, dv\right)dr \\
&\quad + \frac{4D}{n}\int^{t}_{0}A(\Sigma_{r})^{\frac{2-n}{n}}\left(\int_{\Sigma_{r}}{\rm tr}_{\Sigma_{r}}h^{M}\varphi \, dv\right)dr \nonumber \\
&\quad + \frac{4D}{n}\cot\theta\int^{t}_{0}A(\Sigma_{r})^{\frac{2-n}{n}}\left(\int_{\partial\Sigma_{r}}\varphi \, ds\right)dr \nonumber \\
&\leq -\frac{4D}{n}\Theta^{+}(t)\int^{t}_{0}A(\Sigma_{r})^{\frac{2-n}{n}}\left(\int_{\Sigma_{r}}\varphi \, dv\right)dr \nonumber \\
&\quad +\frac{4D}{n}\cot\theta\int^{t}_{0}A(\Sigma_{r})^{\frac{2-n}{n}}\left(\int_{\partial\Sigma_{r}}\varphi \, ds\right)dr, \nonumber
\end{align*}
since our assumption is $n$-convex on $M_{+}$. Then we have
\begin{align} \label{ine 5.12}
2(\Theta^{+}(t))'\beta_{6}(t) &\leq  2\Theta^{+}(t)\beta_{8}(t) -\frac{4D}{n}\Theta^{+}(t)\int^{t}_{0}\beta_{9}(r)dr \\
&\quad + \cot\theta\left(2\Theta^{+}(t)\beta_{7}(t) + \frac{4D}{n}\int^{t}_{0}\beta_{10}(r)dr\right) \nonumber
\end{align}
where $\beta_{9}(t) = A(\Sigma_{t})^{\frac{2-n}{n}}\left(\int_{\Sigma_{t}}\varphi \, dv\right)$ and $\beta_{10}(t) = A(\Sigma_{t})^{\frac{2-n}{n}}\left(\int_{\partial\Sigma_{t}}\varphi \, ds\right)$. If $t = 0$ and $\theta \in (0,\frac{\pi}{2}]$, then we obtain
\begin{align*}
2(\Theta^{+}(0))'\beta_{6}(0) \leq 0.
\end{align*}
If $(\Theta^{+}(0))' < 0$, then $\Theta^{+}(t) < 0$ for $t \in [0,\epsilon)$. This is a contradiction since $\Sigma$ is a weakly outermost. So $(\Theta^{+}(0))' = 0$. Let $\beta(t)$ denote the right hand side of the inequality (\ref{ine 5.12}). Then we get
\begin{align*}
\beta'(0) = \frac{2D}{n}\cot\theta\frac{\beta_{10}(0)}{\beta_{6}(0)}.
\end{align*}
This implies that 
\begin{align*}
(\Theta^{+}(0))'' < \beta'(0) < 0
\end{align*}
if and only if $\theta \neq \frac{\pi}{2}$. Since $\Sigma$ is a weakly outermost, $\theta$ must be $\frac{\pi}{2}$. Therefore, we have
\begin{align*}
(\Theta^{+}(t))' -\Theta^{+}(t)\beta_{11}(t) \leq 0,
\end{align*}
where $\beta_{11}(t) = \frac{\beta_{8}(t) - \frac{2D}{n}\int^{t}_{0}\beta_{9}(r)dr}{\beta_{6}(t)}$. It follows that
\begin{align*}
\left(e^{-\int^{t}_{0}\beta_{11}(r)dr}\Theta^{+}(t)\right)' \leq 0.
\end{align*}
This gives that $\Theta^{+}(t) = 0$ for $t \in [0,\epsilon)$, since $\Sigma$ is a weakly outermost. Hence, the above all inequalities become equalities. So we obtain ${\rm tr}_{\Sigma_{t}}h^{M} = 0$ and $H(t) = 0$ on $\Sigma_{t}$. It follows that $\Theta^{-}(t) = 0$ for $t \in [0,\epsilon)$.

\begin{itemize}
\item[CASE (B).] $D = 0$, $\sigma^{1,0}(\Sigma,\partial\Sigma) \leq 0$, $n$-convex, and $\Sigma$ is an energy-minimizing.
\end{itemize}
By the inequality (\ref{ine 5.9}), we have
\begin{align} \label{ine 5.13}
2(\Theta^{+}(t))'\int_{\Sigma_{t}}\frac{u^{2}_{t}}{\varphi}dv &\leq \int_{\Sigma_{t}}\left(\frac{4(n-1)}{n-2}\vert \nabla^{\Sigma_{t}}u_{t} \vert^{2} + R^{\Sigma_{t}}u^{2}_{t}\right)dv + 2\int_{\partial\Sigma_{t}}H^{\partial\Sigma_{t}}u^{2}_{t}ds \\
&\quad  + 2\Theta^{+}(t)\int_{\Sigma_{t}}{\rm tr}_{g}h^{M}u^{2}_{t}dv + 2(\cot\theta)\Theta^{+}(t)\int_{\partial\Sigma_{t}}u^{2}_{t}ds. \nonumber
\end{align}
Dividing the inequality (\ref{ine 5.13}) by $\left(\int_{\Sigma_{t}}u^{\frac{2n}{n-2}}_{t}dv\right)^{\frac{n-2}{n}}$, then we get
\begin{align*}
(\Theta^{+}(t))'\beta_{6}(t) &\leq \sigma^{1,0}(\Sigma,\partial\Sigma) + \Theta^{+}(t)\left((\cot\theta)\beta_{7}(t) + \beta_{8}(t)\right) \\
&= \Theta^{+}(t)\left((\cot\theta)\beta_{7}(t) + \beta_{8}(t)\right).
\end{align*}
If $t = 0$, then $(\Theta^{+}(0))' \leq 0$ for $t \in [0,\epsilon)$. This implies that $(\Theta^{+}(0))' = 0$, since $\Sigma$ is a weakly outermost. Let $\beta_{12}(t) = \frac{(\cot\theta)\beta_{7}(t) + \beta_{8}(t)}{\beta_{6}(t)}$. Then we obtain
\begin{align*}
\left(e^{-\int^{t}_{0}\beta_{12}(r)dr}\Theta^{+}(t)\right)' \leq 0.
\end{align*}
It follows that $\Theta^{+}(t) = 0$ for all $t \in [0,\epsilon)$. By the energy-minimizing, we have
\begin{align*}
0 \leq E(t) - E(0) = \int^{t}_{0}E'(r)dr = \int^{t}_{0}\left(\int_{\Sigma_{r}}H(r)\varphi \, dv\right)dr.
\end{align*}
Since our assumption is $n$-convex, we deduce that $0 \leq H(t) \leq {\rm tr}_{\Sigma_{t}}h^{M} + H(t) = \Theta^{+}(t) = 0$. So we conclude that $H(t) = 0$ and ${\rm tr}_{\Sigma_{t}}h^{M} = 0$ on $\Sigma_{t}$. This gives $\Theta^{-}(t) = 0$ for $t \in [0,\epsilon)$.

In the case $\theta = \frac{\pi}{2}$, conclusions {\rm (i), (ii), (iii)} and {\rm (iv)}  of Theorem \ref{main high 1-2} (also stated as Theorem \ref{main high 2}) follow from Theorem B in \cite{AM} and the argument in Section 5 of \cite{M2}.

Finally, we prove that $M$ is Ricci flat around $\Sigma$ if $D = 0$, $Ric^{M} = \frac{\mu}{n+1}g$, $H^{\partial M} \geq 0$, and $\theta \neq \frac{\pi}{2}$. If we follow the final part of the proof of Theorem \ref{main thm 1-2} (Theorem \ref{main thm 2}), we obtain the following chain of inequalities:
\begin{align*}
0 \leq \int_{\partial\Sigma_{t}}\frac{\partial\varphi}{\partial\nu_{t}}ds = \int_{\Sigma_{t}} \Delta_{\Sigma_{t}}\varphi \, dv = -\int_{\Sigma_{t}}\left(\frac{\vert \nabla^{\Sigma_{t}}\varphi\vert^{2}}{\varphi}  + \left(\vert J \vert + \frac{1}{4}\vert \chi^{-}_{t} \vert^{2} \right)\varphi \right)dv \leq 0.   
\end{align*}
It follows that all inequalities must be equalities. Therefore, we have $\vert J \vert = 0$ on $\Sigma_{t}$ and $H^{\partial M} = 0$ along $\partial\Sigma_{t}$. If, in addition, $M$ satisfies $Ric^{M} = \frac{\mu}{n+1}g$ and $D = 0$, then it follows that $\mu = 0$ and hence $R^{\Sigma} = 0$ in an outer neighborhood of $\Sigma$. Thus, we conclude that $M$ is Ricci flat with $H^{\partial M} = 0$ in an outer neighborhood of $\Sigma$.

\end{proof}

\end{document}